\documentclass[12pt,letterpaper]{amsart}
\usepackage[letterpaper, margin=3cm]{geometry}

\usepackage[utf8]{inputenc}
\usepackage[T1]{fontenc}

\usepackage[usenames, dvipsnames]{color}
\usepackage{graphicx}
\usepackage{amssymb}
\usepackage{amsmath}
\usepackage{amsthm}
\usepackage{mathtools}
\usepackage{mathrsfs}
\usepackage[hidelinks]{hyperref}
\usepackage{tikz}
\usepackage{xfrac}
\usepackage[super]{nth}

\mathtoolsset{showonlyrefs}

\newcommand{\C}{\mathbb C}

\newcommand{\D}{\mathbb D}

\newcommand{\R}{\mathbb R}
\newcommand{\N}{\mathbb N}
\newcommand{\Z}{\mathbb Z}

\newcommand{\de}{\, \mathrm{d}}

\newcommand{\del}{\partial}

\newcommand{\pardiff}[2]{\frac{\partial #1}{\partial #2}}
\newcommand{\der}[2]{\frac{\de #1}{\de #2}}
\newcommand{\norm}[1]{\Vert #1 \Vert}
\newcommand{\abs}[1]{| #1 |}
\newcommand{\dual}[2]{\langle #1 , #2 \rangle}

\newcommand{\eps}{\varepsilon}
\newcommand{\ot}{\leftarrow}

\DeclareMathOperator{\intr}{int}

\DeclareMathOperator{\vol}{vol}

\DeclareMathOperator{\Id}{Id}
\DeclareMathOperator{\Tr}{Tr}

\newtheorem{thm}{Theorem}[section]
\newtheorem{prop}[thm]{Proposition}
\newtheorem{lem}[thm]{Lemma}
\newtheorem{coro}[thm]{Corollary}

\theoremstyle{definition}

\newtheorem{rem}[thm]{Remark}

\usepackage{amsmath}
\usepackage{tikz-cd}

\author{Simon St-Amant}
\title[Stability estimate for the broken X-ray in Minkowski space]{Stability estimate for the broken non-abelian X-ray transform in Minkowski space}
\address{Department of Pure Mathematics and Mathematical Statistics, University of Cambridge, Cambridge CB3 0WB, UK}
\email{sas242@cam.ac.uk}
\date{January 21, 2022}

\begin{document}

\begin{abstract}
We study the broken non-abelian X-ray transform in Minkowski space. This transform acts on the space of Hermitian connections on a causal diamond and is known to be injective up to an infinite-dimensional gauge. We show a stability estimate that takes into account the gauge, leading to a new proof of the transform's injectivity. Our proof leads us to consider a special type of connections that we call light-sink connections. We then show that we can consistently recover a light-sink connection from noisy measurement of its X-ray transform data through Bayesian inversion.
\end{abstract}

\maketitle

\section{Introduction and main results}

We start by defining the broken non-abelian X-ray transform and provide the motivation for its study. We then state our main results. Sections \ref{sec:stability} and \ref{sec:gauge} contain the proofs of those results.

\subsection{The broken non-abelian X-ray transform}

Consider the causal diamond in Minkowski space $(\R^{1+3}, -dt^2 + dx_1^2 + dx_2^2 + dx_3^2)$ given by
\begin{equation}
	\D := \{(t,x) \in \R^{1+3} : \abs{x} \leq t+1, \abs{x} \leq 1 - t\}.
\end{equation}
The origin's world line is $\mathcal{O} = (-1,1) \times \{(0,0,0)\} \subset \D$. For $0 < \varepsilon \leq \varepsilon_0 < 1/2$, consider the $\varepsilon$-neighbourhood of $\mathcal{O}$
\begin{equation}
	\mho_\varepsilon := \{(t,x) \in \intr\D : \abs{x} < \varepsilon\}.
\end{equation}
We implicitly write $\mho$ for $\mho_{\varepsilon_0}$ and write $\mho_\varepsilon$ whenever we want to emphasise the dependence on $\varepsilon$. Given $x,y \in \D$, we write $x < y$ if there is a future-pointing causal curve from $x$ to $y$. We also write $(x,y) \in \mathbb{L}$ if $x < y$ and there is a lightlike geodesic from $x$ to $y$.

Recall that a line segment $\gamma : [0,T] \to \R^{1+3}$, $\gamma(s) = x + sv$ for $x \in \R^{1+3}$, $v = (v_0, v_1, v_2, v_3)$ is a lightlike geodesic if
\begin{equation}
	v_0^2 = v_1^2 + v_2^2 + v_3^2
\end{equation}
and that it is future-pointing if $v_0 > 0$ and past-pointing if $v_0 < 0$. We say that $\gamma$ is parametrised by arc length if $\abs{v}_{\R^4} = 1$. The set of points $y \in \R^{1+3}$ such that there is a future-pointing (past-pointing) lightlike geodesic from $x$ to $y$ is called the future (past) light cone at $x$. Hence, $(x,y) \in \mathbb{L}$ if and only if $y$ is in the future light cone of $x$, or equivalently, $x$ is in the past light cone of $y$.

We will work with Hermitian connections on the trivial bundle $\D \times \C^{n}$. Such a connection $A$ is a $\mathfrak{u}(n)$-valued one-form on $\D$ and we can write it as
\begin{equation}
	A = A_0 dt + A_1 dx^1 + A_2 dx^2 + A_3 dx^3
\end{equation}
for some matrix fields $A_i \in C^\infty(\D, \mathfrak{u}(n))$. We denote the set of Hermitian connections on $\D$ by $\mathscr{U}$. A connection induces a covariant derivative on functions $f : \D \to \C^n$ given by $d_A f = df + Af$. Given a smooth curve $\gamma:[0,T] \to \D$, the parallel transport isomorphism $P^A_\gamma : \C^{n} \to \C^{n}$ is given by the solution of the matrix ODE
\begin{equation}\label{eq:paralleltransport}
	\begin{cases}
		\dot{U}(t) + A(\dot{\gamma}(t))U(t) = 0, \\
		U(0) = \Id,
	\end{cases}
\end{equation}
at time $T$. Hence, the parallel transport of a vector $v \in \C^n$ along $\gamma$ is $P^A_\gamma v := U(T)v$. One can check that $P^A_\gamma$ does not depend on the parametrisation of $\gamma$ and that it takes values in $U(n)$ since $A$ is Hermitian. Given $x, y \in \D$, we denote by $P^A_{y\leftarrow x}$ the parallel transport from $x$ to $y$ along the straight line between the two points. The notation is chosen as to behave nicely with compositions.

We can now define the broken non-abelian X-ray transform. In \cite{chen2019detection} and \cite{chen2021inverse}, they define it as follows. Consider the set
\begin{equation}
	\mathbb{S}^+(\mho) := \{(x,y,z) \in \D^3: (x,y), (y,z) \in \mathbb{L}, x < y < z \text{ with } x,z \in \mho, y \not\in \mho\}.
\end{equation}
This set is comprised of light rays starting from $x \in \mho$ that exit $\mho$ and break at $y \not\in \mho$ before returning to $\mho$ at $z$. We denote by
\begin{equation}
	\mho^X := \bigcup_{(x,y,z) \in \mathbb{S}^+(\mho)} \{x\} \quad \text{and} \quad \mho^Z := \bigcup_{(x,y,z) \in \mathbb{S}^+(\mho)} \{z\}
\end{equation}
the sets of values that $x$ and $z$ can take in $\mho$, respectively. It is important to note that neither $\mho^X$ or $\mho^Z$ cover $\mho$, but that $\mho = \mho^X \cup \mho^Z$. Given a Hermitian connection $A$ as above, its broken non-abelian X-ray transform is
\begin{equation}
	S^A_{z \ot y \ot x} := P^A_{z \ot y} P^A_{y \ot x}, \qquad (x,y,z) \in \mathbb{S}^+(\mho).
\end{equation}

We are interested in recovering the connection $A$ from its scattering data $S^A$. However, the map $A \mapsto S^A$ is not injective as it has a gauge given by the following right group action. For $\varphi \in C^\infty(\D, U(n))$, we denote
\begin{equation}
	A \triangleleft \varphi := \varphi^{-1} d\varphi + \varphi^{-1} A \varphi.
\end{equation}
The next proposition, whose proof is straightforward, states that the action of $\varphi$ on the connection amounts to a conjugation of the parallel transports.

\begin{prop}\label{prop:gaugeparallel}
Let $A$ be a connection on $\D$ and let $\varphi \in C^{\infty}(\D, U(n))$. Then, for any smooth curve $\gamma : [0,T] \to \D$,
\begin{equation}
	P^{A \triangleleft \varphi}_\gamma = \varphi(\gamma(T))^{-1} P^A_\gamma \varphi(\gamma(0)).
\end{equation}
In particular, if $\varphi\vert_{\mho} = \Id$, then
\begin{equation}
	S^{A \triangleleft \varphi}_{z \ot y \ot x} = S^A_{z \ot y \ot x}
\end{equation}
for all $(x,y,z) \in \mathbb{S}^+(\mho)$.
\end{prop}

Therefore, the scattering data of $A \triangleleft \varphi$ coincides with that of $A$ whenever $\varphi$ is in the gauge group
\begin{equation}
	\mathscr{G} := \{\varphi \in C^\infty(\D, U(n)) : \varphi\vert_{\mho} = \Id\}.
\end{equation}
This natural obstruction to recovering $A$ from $S^A$ turns out to be the only one. Indeed, it is shown in \cite[Theorem 5]{chen2019detection} that Hermitian connections $A$ and $B$ share the same scattering data if and only if they are in the same gauge orbit, that is, there exists $\varphi \in \mathscr{G}$ such that $B = A \triangleleft \varphi$.

Our goal is to find a stability estimate relating the scattering data of two connections $A$ and $B$ with some measure of distance between them in a gauge invariant way. In other words, we want to show that $A$ and $B$ must be relatively similar whenever $S^A$ and $S^B$ are close.

\subsection{The non-abelian X-ray transform and broken Radon transform}

The usual non-abelian X-ray transform assigns to a matrix field $A \in C^\infty(\R^d \times \mathbb{S}^{d-1}, \C^{n \times n})$ the scattering data map
\begin{equation}
	(x,\theta) \in \R^d \times \mathbb{S}^{d-1} \mapsto \lim_{x \to \infty} \psi^+(x + s\theta, \theta) \in \C^{n \times n}
\end{equation}
where $\psi^+$ is the unique solution of the transport equation 
\begin{equation}
	\sum_{i=1}^{d} \theta_i \del_{x_i} \psi + A(x,\theta) \psi = 0, \quad x \in \R^d, \quad \theta \in \mathbb{S}^{d-1},
\end{equation}
such that
\begin{equation}
\lim_{s \to -\infty} \psi^+(x + s\theta, \theta) = \Id.
\end{equation}
Given that $A$ decays sufficiently fast as $\abs{x} \to \infty$, the transform is well-defined and one can ask whether it is possible to recover $A$ from the scattering data. The non-abelian X-ray transform has been studied extensively in the last 20 years and has applications in many different types of tomographies, such as single-photon emission computed tomography or neutron polarisation tomography. See \cite{novikov20195} for a recent survey on the non-abelian X-ray transform and its applications.

The non-abelian X-ray transform has also been studied on simple surfaces \cite{paternain2020non, monard2021consistent} and compact manifolds with strictly convex boundary \cite{bohr2021stability} where the transport equation is now solved along unit-speed geodesics with endpoints on the boundary of the manifold. For more details and background on the two-dimensional problem, see \cite{GIP2D}.

When $n = 1$, the broken non-abelian X-ray transform is also called the broken-ray Radon transform. In \cite{brokenray}, they consider the broken-ray Radon transform with rays breaking at a fixed angle within a slab and provide an inversion formula. The broken-ray Radon transform has applications in optical tomography, see \cite{opticaltomography} for a survey. The V-line Radon transform \cite{vlineinversion, vlinegeneralization} is another example of an inverse problem making use of broken rays and has applications in imaging.

\subsection{Physical motivation}

The broken non-abelian X-ray transform has been introduced in \cite{chen2019detection} where they began to analyse inverse problems for the Yang-Mills-Higgs equations. They show that one can recover a Hermitian connection $A$ from the source-to-solution map $L_A$ taking a source $f \in C^4_c(\mho, \C^n)$ to
\begin{equation}
	L_A f = \phi \vert_\mho
\end{equation}
where $\phi$ solves
\begin{equation}\label{eq:phi}
\begin{cases}
	\square_A\phi + \abs{\phi}^2 \phi = f & \text{in } (-1,2) \times \R^3, \\
	\phi\vert_{t < -1} = 0.
\end{cases}
\end{equation}
Here $\square_A$ is the connection wave operator given by $\square_A = d_A^* d_A$. Note that when $A = 0$, we recover the usual wave operator $\square = \del_t^2 - \Delta$. The map $L_A$ is well-defined as long as $f$ is sufficiently small. They show that the maps $L_A$ and $L_B$ agree if and only if $A$ and $B$ are gauge equivalent. To do so, they first show that $L_A$ determines the broken non-abelian X-ray transform $S^A_{z \ot y \ot x}$ for all $(x,y,z) \in \mathbb{S}^+(\mho)$. Injectivity up to gauge of $L_A$ then follows from that of the broken X-ray transform.

To determine $S^A_{z \ot y \ot x}$ from the source-to-solution map $L_A$, they construct a source of the form
\begin{equation}
	f = \epsilon_1 f_1 + \epsilon_2 f_2 + \epsilon_3 f_3
\end{equation}
where each $f_j$ is a conormal distribution supported near $x \in \mho$. Let $\phi$ be the solution of \eqref{eq:phi} corresponding to such an $f$. The functions $\del_{\epsilon_j}\phi\vert_{\epsilon_j=0}$ satisfy a wave equation and, when the sources are chosen carefully, can produce an artificial source at $y$ which emits a singular wave front that reaches $z$. This interaction is encoded in the operator $f \mapsto \del_{\epsilon_1}\del_{\epsilon_2}\del_{\epsilon_3} \phi \vert_{\epsilon = 0}$, whose principal symbol determines $S^A_{z \ot y \ot x}$. The creation of an artificial source is only possible thanks to the nonlinearity in \eqref{eq:phi} and shows how one can exploit nonlinearities in an advantageous way, similar to what is shown in \cite{kurylev2018inverse}.

\subsection{Statistical motivation}\label{sec:statsmotivation}

The second motivation for considering the broken non-abelian X-ray transform is to use it as an example for dealing with injectivity issues that arise in the study of Bayesian inverse problems. We give a short summary to the Bayesian approach to solving inverse problems, as introduced in \cite{stuart2010inverse}.

For some mapping $\mathcal{G} : \Theta \to Y$ between Banach spaces, and $y \in Y$, we wish to find $\theta \in \Theta$ such that
\begin{equation}
	y = \mathcal{G}(\theta).
\end{equation}
Let us take $Y = L^2_{\lambda}(\mathcal{X}, \mathcal{V})$, the set of square-integrable functions on a probability space $(\mathcal{X}, \lambda)$ with values in a finite-dimensional normed space $\mathcal{V}$. Rather than working with the whole infinite-dimensional $L^2$ space, we discretise it by considering the following regression model which mimics the setting of an experiment. Let $(X_i)_{i=1}^N$ be i.i.d.\ random variables on $\mathcal{X}$ with distribution $\lambda$. These random variables correspond to experimental measurement of $\mathcal{G}_\theta = \mathcal{G}(\theta)$ with input $X_i$. Such measurements come with experimental noise that we model through the random variables
\begin{equation}
	V_i = \mathcal{G}_\theta(X_i) + \mathcal{E}_i, \quad i = 1, \dots, N,
\end{equation}
where the $\mathcal{E}_i$ are i.i.d.\ standard Gaussian variables on $\mathcal{V}$, independent of the $X_i$.

In our setting, the set $\Theta$ could be the set of Hermitian connections $A$ on $\D$, $Y$ the set of matrix fields on $\mathbb{S}^+(\mho)$ and $\theta \mapsto \mathcal{G}_\theta$ the mapping that sends a connection $A$ to its scattering data $S^A$. Each $X_i$ then amounts to a random choice of path $z \ot y \ot x$ in $\mathbb{S}^+(\mho)$ and $V_i$ a noisy version of $S^A_{z \ot y \ot x}$.

Let $D_N = \{(V_i, X_i) : i = 1, \dots, N\} \subset (\mathcal{V} \times \mathcal{X})^N$ be the full data vector and let $P^N_\theta$ be its law. By making a choice of prior $\Pi$ on the parameter space $\Theta$, Bayes' rule yields a posterior distribution on $\Theta$ given the data $D_N$. For a Borel set $O \subset \Theta$, it is given by
\begin{equation}
	\Pi^N(\theta \in O \vert D_N) = \frac{\int_O e^{\ell_N(\theta)} \Pi(d\theta)}{\int_\Theta e^{\ell_N(\theta)} \Pi(d\theta)}
\end{equation}
where the log-likelihood is, up to additive constants,
\begin{equation}
	\ell_N(\theta) = \ell_N(\theta \vert D_N) = -\frac{1}{2} \sum_{i=1}^N \abs{V_i - \mathcal{G}_\theta(X_i)}_\mathcal{V}^2.
\end{equation}

One can study how the posterior distribution $\Pi^N$ behaves when $N$ gets large. If we suppose there exists a unique underlying parameter $\theta_\star \in \Theta$ from which the observations are made, we would want the posterior distribution to concentrate around $\theta_\star$ (see \cite[Chapter 7.3]{nicklgine} or \cite{ghosal}), that is, we would want that
\begin{equation}\label{eq:posteriorcontraction}
	\Pi^N(\norm{\theta - \theta_\star}_{L^2(\Theta)} > \delta_N \vert D_N) = o_{P^N_{\theta_\star}}(1)
\end{equation}
as $N \to \infty$ for some sequence $\delta_N \to 0$ that dictates the rate of convergence. Following substantial developments in the field, one should then get a good estimator $\hat{\theta}$ for $\theta_\star$ by computing the expectation of the posterior distribution $\Pi^N$ through MCMC sampling. Depending on the inverse problem, can we get estimates such as \eqref{eq:posteriorcontraction} and can we guarantee that the posterior mean indeed converges to $\theta_\star$, legitimising Bayesian inversion? This question has been studied for a range of different inverse problems and is an active area of research, see \cite{monard2021consistent} as well as \cite{abraham2020statistical, bohr2021stability, matteo2020} for some examples.

However, in the case of the broken non-abelian X-ray transform, the map $\mathcal{G} : A \mapsto S^A$ is not injective and so the true underlying parameter is not uniquely identifiable. Indeed, all connections in the same $\mathscr{G}$-orbit yield the same scattering data. Moreover, these orbits are all infinite dimensional. Can we still find a way to get a meaningful candidate for $A$ from samples of $S^A$ through the framework of Bayesian inverse problems?

The first approach one could use to deal with injectivity issues is as follows. Let us assume, as is our case, that a group $G$ acts on $\Theta$ and that $\mathcal{G}$ is injective up to the action of $G$. This means that for every $g \in G$ and $\theta \in \Theta$, we have $\mathcal{G}(\theta \triangleleft g) = \mathcal{G}(\theta)$ and that $\mathcal{G}(\theta_1) = \mathcal{G}(\theta_2)$ if and only if there is $g' \in G$ with $\theta_1 = \theta_2 \triangleleft g'$. Then $\mathcal{G}$ naturally induces an injective map on the quotient space
\begin{equation}
	\tilde{\mathcal{G}} : \sfrac{\Theta}{G} \to Y.
\end{equation}
One could try to prove statistical guarantees for this map. However, as settings such as the present one where $\mathcal{G}$ is non-linear, the quotient space $\sfrac{\Theta}{G}$ is intractable as it is unclear how one would parametrise the equivalence classes. What one needs is a choice of representative for each class in the quotient, that is, a continuous map $s : \sfrac{\Theta}{G} \to \Theta$ such that the following diagram commutes.
\begin{center}
\begin{tikzcd}
\Theta \arrow[r, "\mathcal{G}"]                                           & Y \\
\sfrac{\Theta}{G} \arrow[u, "s"] \arrow[ru, "\tilde{\mathcal{G}}"'] &  
\end{tikzcd}
\end{center}
The existence of $s$ is nontrivial and it is often the case that such a lift simply does not exist, see \cite{singer} for examples where topological obstructions prevent its existence. And even if $s$ exists, it might only be theoretical and not correspond to an explicit choice (not constructive or numerically computable). Hence, we need a new approach that is adapted to the problem we want to consider.

What we will end up doing is finding another group $H$ of which $G$ is a proper subgroup and for which we can find an explicit section $s_H : \sfrac{\Theta}{H} \to \Theta$. Although the forward map will not be invariant under the action of $H$, our stability estimates will. Those same estimates will guarantee that the forward map is injective when restricted to the image of $s_H$. We will then show that we can use Bayesian inversion to solve this restricted problem. Finally, through some choice of extension operator, we will show that, from the solution to the restricted problem, we can recover an element that is $G$-equivalent to the true solution $\theta_\star$. See the discussion after Proposition \ref{prop:HtoG} for more details.

\subsection{Definitions and notation}

Before presenting the main results, we use this section to gather some notation and additional definitions that will be used throughout.

Unlike in \cite{chen2019detection}, we will not consider all paths in $\mathbb{S}^+(\mho)$. We will mostly consider two types of paths that we refer to as \textbf{past-determined} and \textbf{future-determined} paths. A past-determined path is a path of the form $z \ot y \ot x_y$ for $(x_y, y, z) \in \mathbb{S}^+(\mho)$ where $x_y$ is the unique point such that $(x_y, y) \in \mathbb{L}$ and $x_y \in \mathcal{O}$, that is, $x_y = (t,0,0,0)$ for some $t \in [-1,1]$. Similarly, a future-determined path is a path of the form $z_y \ot y \ot x$ for $(x,y,z_y) \in \mathbb{S}^+(\mho)$ where now $z_y$ is the unique point on $\mathcal{O}$ such that $(y, z_y) \in \mathbb{L}$. We denote the corresponding scattering data as
\begin{equation}
	S^A_{z \ot y \ot x_y} \quad \text{and} \quad S^A_{z_y \ot y \ot x}.
\end{equation}
Hence, the wiggle room in $\mho$ will only be used to move $x$ in $\mho^X$ or $z$ in $\mho^Z$, but not both.

\begin{rem}
It is not sufficient to only consider paths that are both past-determined and future-determined, that is, paths of the form $z_y \ot y \ot x_y$. Indeed, 	in polar coordinates $(t,r,\vartheta,\phi)$, the tangent vector along the path $z_y \ot y$ is $\frac{1}{\sqrt{2}}\left(\pardiff{}{t} - \pardiff{}{r}\right)$ while the tangent vector along the path $y \ot x_y$ is $\frac{1}{\sqrt{2}}\left(\pardiff{}{t} + \pardiff{}{r}\right)$. Hence, the angular components of the connection play no role in the forward problem for such paths.
\end{rem}

Any future-determined path can be identified by its break point $y \in \D \setminus \mho$ and its first endpoint $x$ which lies in the intersection between $\mho^X$ and the past light cone of $y$. We can represent the admissible future-determined paths as a $B_3$-bundle $\pi : \mathcal{F}^X \to \D \setminus \mho$, where $B_3$ stands for the unit ball in $\R^3$. For every $y \in \D\setminus\mho$, the fibre is given by
\begin{equation}
	\mathcal{F}^X_y = \{x \in \mho^X : (x,y) \in \mathbb{L}\} \cong B_3.
\end{equation}
Similarly, the set of admissible past-determined paths can be represented through the $B_3$-bundle $\pi:\mathcal{F}^Z \to \D\setminus\mho$ with fibre
\begin{equation}
	\mathcal{F}^Z_y = \{z \in \mho^Z : (y,z) \in \mathbb{L}\} \cong B_3.
\end{equation}
We write $\mathcal{F}^X_\varepsilon$ or $\mathcal{F}^Z_\varepsilon$ whenever we want to emphasise the dependence of the bundles on $\varepsilon$ through $\mho_\varepsilon$.

For two points $x$ and $y$ in $\D$, let $\gamma_{y \ot x} : [0,T] \to \D$ be the straight line from $x$ to $y$ parametrised by its (Euclidean) arc length. We denote
\begin{equation}
	v_{y \ot x} := \dot{\gamma}_{y \ot x}(T),
\end{equation}
that is, $v_{y \ot x}$ is the unit length vector pointing from $x$ to $y$, but based at $y$. For a function $\Phi : \D^2 \setminus \Gamma \to \C^n$ where $\Gamma$ is the diagonal of $\D^2$,  we define the differential operator
\begin{equation}\label{eq:delyx}
	(\del\Phi)(x,y) = \del_{y \ot x} \Phi(x,y) := \der{}{t} \Phi(x, y + tv_{y \ot x}) \bigg\vert_{t = 0}.
\end{equation}
Note that if $(x,y) \in \mathcal{F}^X$ and the domain of $\Phi$ is $\mathcal{F}^X$, the operators $\del_{y \ot x}$ and $\del_{x \ot y}$ are both well-defined since $x \in \mathcal{F}^X_{y + tv_{y \ot x}}$ and $x + tv_{x \ot y} \in \mathcal{F}^X_y$ whenever $x \in \mathcal{F}^X_y$ and $t$ is sufficiently small. One can see $\del_{y \ot x}$ and $\del_{x \ot y}$ as horizontal and vertical vector fields on $\mathcal{F}^X$, respectively. Similarly, $\del_{y \ot z}$ and $\del_{z \ot y}$ are well defined operators if $(y,z) \in \mathcal{F}^Z$ and the domain of $\Phi$ is $\mathcal{F}^Z$.

\begin{figure}
\includegraphics[scale=0.75, viewport = 0 70 440 427]{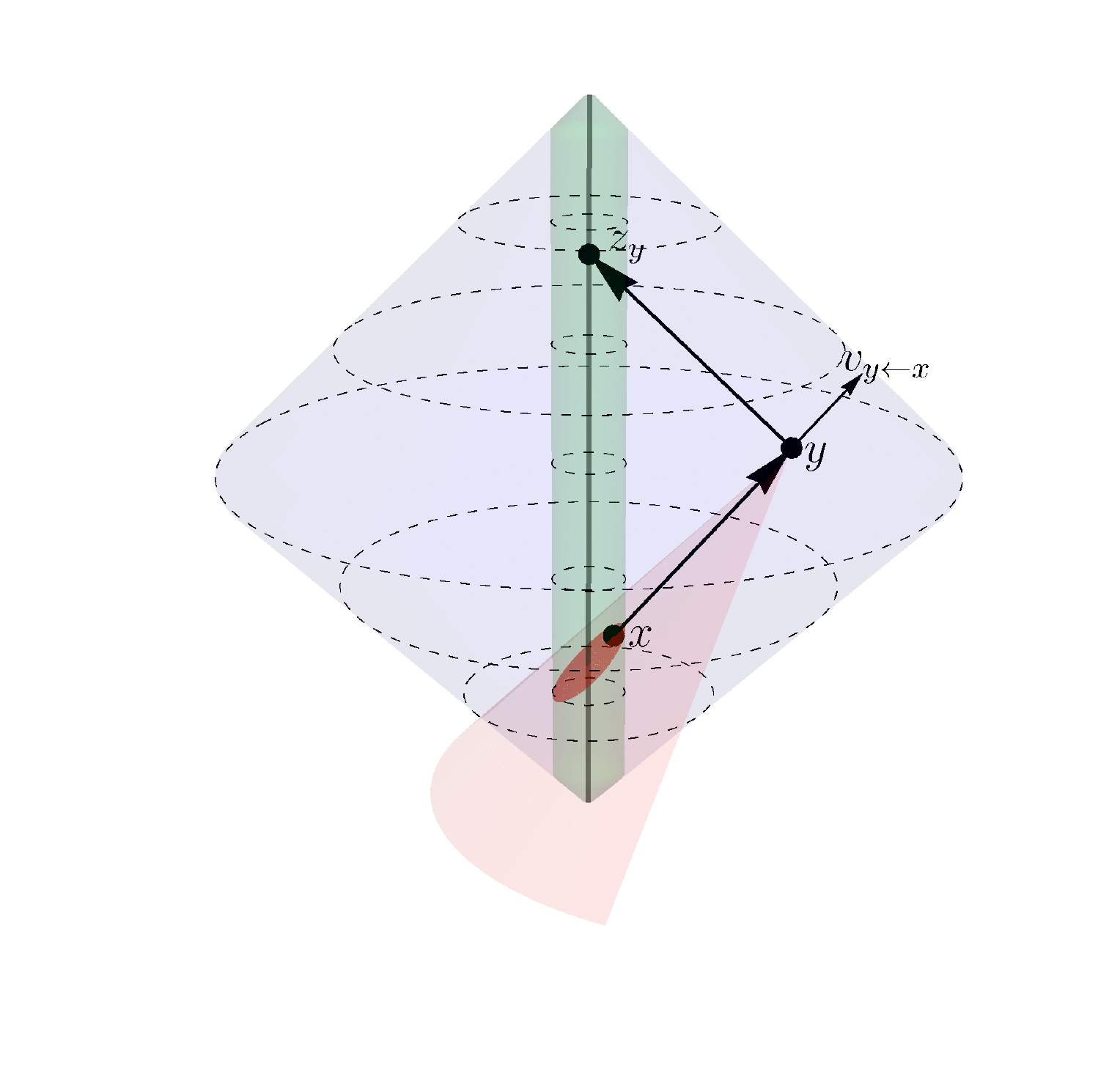}
\caption{Setting for the broken non-abelian X-ray transform in $\R^{1+2}$ for a future-determined path. The point $y$ lies inside the causal diamond (in blue), but outside the set $\mho$ (in green). The point $x$ can take values in the fibre $\mathcal{F}^X_y$ which is given by the intersection of $\mho$ and the past light cone at $y$ (in red). The point $z_y$ is always taken on the origin's world line and is uniquely determined by $y$. The vector $v_{y \ot x}$ is based at $y$ and points in the direction coming from $x$.}
\end{figure}

We define the $L^2$-norm of a function $\Phi : \mathcal{F}^X \to \C^n$ as
\begin{equation}
	\norm{\Phi}_{L^2(\mathcal{F}^X)} = \left(\int_{\D\setminus\mho} \int_{\mathcal{F}^X_y}\abs{\Phi(x,y)}^2 \de x \de y\right)^{1/2}
\end{equation}
where $\de x$ is the natural measure on $\mathcal{F}^X_y$ induced by Euclidean space. Note that this norm scales down as $\varepsilon$ goes to $0$ at a rate of $\varepsilon^3$.

Given a linear map $\mathcal{T}$ from $\R^m$ to $\C^n$, we will denote its operator norm as
\begin{equation}
	\norm{\mathcal{T}} := \sup_{v \in \R^m\setminus\{0\}} \frac{\abs{\mathcal{T}(v)}}{\abs{v}}
\end{equation}
where $\abs{\,\cdot\,}$ denotes the usual norm on $\R^m$ or $\C^n$. This induces a pointwise norm on $\C^{n\times n}$-valued one-forms $\omega$ on $\D$ at any given point $y \in \D$ by seeing $\omega_y$ as a mapping from $\R^4$ to $\C^{n^2}$. Note that we then have $\abs{\omega_y (v)} = \Tr([\omega_y \omega_y^*](v))$ and so $\norm{\omega_y}$ is invariant under the action of $U(n)$. This also induces an $L^2$-norm on the space of one-forms by
\begin{equation}
	\norm{\omega}_{L^2(\D)} = \left(\int_{\D} \norm{\omega_y}^2 \de y\right)^{1/2}.
\end{equation}

Given two connections $A$ and $B$ and a matrix field $Q \in C^\infty(\D, \C^{n\times n})$, we define $E(A,B) \in C^{\infty}(\D, \mathrm{End}(\C^{n\times n}))$ by $E(A,B) Q = AQ - QB$. If $A$ and $B$ are Hermitian, then so is $E(A,B)$, in the sense that $E(A,B)^* = -E(A,B)$.

\subsection{Main results}

We state our results only for future-determined paths, but equivalent statements hold for past-determined paths by seeing $[S^A_{z \ot y \ot x_y}]^{-1}$ as a future-determined path. We will first show the stability estimates below for the values of a connection inside and outside $\mho$.
\begin{thm}\label{thm:estimatein}
Let $A$ and $B$ be Hermitian connections on $\D$. There is a constant $C > 0$ independent of $\varepsilon$ such that
\begin{equation}
	\norm{A - B}_{L^2(\mho^X_\varepsilon)} \leq C \norm{\del_{x \ot y} (S^A_{z_y \ot y \ot x} [S^B_{z_y \ot y \ot x}]^{-1})}_{L^2(\mathcal{F}^X_\varepsilon)}.
\end{equation}
\end{thm}

\begin{thm}\label{thm:estimateout}
Let $A$ and $B$ be Hermitian connections. There exists a smooth function $p \in \C^\infty(\D, \C^{n \times n})$ vanishing on $\mathcal{O}$ and $C > 0$ such that for all $0 < \varepsilon < \varepsilon_0$,
\begin{equation}\label{eq:estimateout}
	\norm{A - B - d_{E(A,B)} p}_{L^2(\D\setminus\mho_\varepsilon)} \leq \frac{C}{\varepsilon^4} \norm{\del_{y \ot x}([S^A_{z_y \ot y \ot x}]^{-1} S^B_{z_y \ot y \ot x})}_{L^2(\mathcal{F}_\varepsilon^X)}.
\end{equation}
\end{thm}

By combining both theorems we can get a new proof of the injectivity (up to the gauge $\mathscr{G}$) of the broken non-abelian X-ray transform.
\begin{coro}
Let $A$ and $B$ be Hermitian connections. Then $S^A$ and $S^B$ agree for all past-determined and future-determined paths if and only if $A$ and $B$ are gauge equivalent.	
\end{coro}

\begin{proof}
Since the scattering data of $A$ and $B$ agree for all future-determined paths, Theorem \ref{thm:estimatein} implies that $A$ and $B$ must agree on $\mho^X$. By the same estimate for past-determined paths, the connections must also agree on $\mho^Z$ and hence they agree on $\mho$. Theorem \ref{thm:estimateout} yields $p \in C^{\infty}(\D, \C^{n \times n})$ such that
\begin{equation}\label{eq:B}
	B = A - d_{E(A,B)} p = A - dp - Ap + pB
\end{equation}
on $\D\setminus\mho$. Let $\varphi = \Id - p$. As the proof of Theorem \ref{thm:pointstability} will reveal, $\varphi$ takes values in $U(n)$ since actually $\varphi = P^A_{y \ot z_y} P^B_{z_y \ot y}$. We can rewrite \eqref{eq:B} as
\begin{equation}
	\varphi B = A \varphi + d \varphi
\end{equation}
and so $B = A \triangleleft \varphi$. It follows that $A$ and $B$ are gauge equivalent since they agree on $\mho$ and so $\varphi\vert_\mho = \Id$. The converse implication is the statement of Proposition \ref{prop:gaugeparallel}.
\end{proof}

This can be seen as a partial data result improving on Theorem 5 in \cite{chen2019detection} as we only considered past-determined and future-determined paths. It also suggests that always taking such paths might be a more efficient problem to study.

Both stability estimates are invariant under the action of $\mathscr{G}$, but they are also invariant under the action of the bigger group
\begin{equation}
	\mathscr{H} := \{\varphi \in C^\infty(\D, U(n)) : \varphi \vert_{\mathcal{O}} = \Id \}.
\end{equation}
In fact, we can rewrite the left-hand side of \eqref{eq:estimateout} in a way that highlights this.

\begin{thm}\label{thm:gauge}
Let $p$ be as in Theorem \ref{thm:estimateout}. Then for all $y \in \D\setminus\mho$,
\begin{equation}
	\norm{(A - B - d_{E(A,B)} p)_y} = \norm{(A \triangleleft P^A_{y \ot z_y})_y - (B \triangleleft P^B_{y \ot z_y})_y}.
\end{equation}	
\end{thm}

Hence, this defines a distance between the connections $A$ and $B$ that is invariant under the action of $\mathscr{H}$, and so gauge independent as $\mathscr{G} \subset \mathscr{H}$. With a little bit of work, we can combine this expression with Theorem \ref{thm:estimateout} to get the following $H^1$ estimate.

\begin{coro}\label{coro:H1}
Let $A$ and $B$ be Hermitian connections. There exists a constant $C > 0$ such that
\begin{equation}
	\norm{(A \triangleleft P^A_{y \ot z_y}) - (B \triangleleft P^B_{y \ot z_y})}_{L^2(\D \setminus \mho_\varepsilon)} \leq \frac{C \Psi(A,B)}{\varepsilon^4} \norm{S^A - S^B}_{H^1(\mathcal{F}_\varepsilon^X)}
\end{equation}
where 
\begin{equation}
\Psi(A,B) = 1 + \min\left\{\norm{F_A}_{L^\infty(\D)} + \norm{A(\del_t)}_{L^\infty(\mathcal{O})}, \norm{F_B}_{L^\infty(\D)} + \norm{B(\del_t)}_{L^\infty(\mathcal{O})}\right\}
\end{equation}
and $F_A$ is the curvature two-form of $A$.
\end{coro}

Theorem \ref{thm:gauge} also suggests we should naturally try to fix the gauge by considering connections such that $A \triangleleft P^A_{y \ot z_y} = A$. We call them \textbf{light-sink} connections. They form a linear space and we can characterise them, see Proposition \ref{prop:lightsink}.

\begin{prop}\label{prop:rho}
Every connection $A$ is $\mathscr{H}$-equivalent to a unique light-sink connection and the map $\rho : \mathscr{U}/\mathscr{H} \to \mathscr{U}$,
\begin{equation}
	\rho([A]) := A \triangleleft P^A_{y \ot z_y}
\end{equation}
is well-defined.
\end{prop}

The map $\rho$ is almost a fixing of the gauge. Contrary to that of $\mathscr{G}$, the action of $\mathscr{H}$ on $\mathscr{U}$ does not preserve the scattering data. Therefore, the map $\rho$ does not define a lift as we defined it in Section \ref{sec:statsmotivation}. Nonetheless, if a light-sink connection $A$ is $\mathscr{H}$-equivalent to another connection $B$, we can use their scattering data and the map $\rho$ to make them gauge equivalent.

\begin{prop}\label{prop:HtoG}
Let $A$ be a light-sink connection and let $B$ be a Hermitian connection such that
\begin{equation}\label{eq:HtoG}
	A = B \triangleleft P^B_{y \ot z_y}.
\end{equation}
From the past-determined and future-determined scattering data of $A$ and $B$, we can find a map $\Phi \in \mathscr{H}$ such that $A \triangleleft \Phi$ and $B$ are gauge equivalent (with respect to $\mathscr{G}$).
\end{prop}

The map $\Phi$ is defined up to an extension operator
\begin{equation}
	\mathcal{E} : C^\infty(\mho, U(n)) \hookrightarrow C^\infty(\D, U(n)).
\end{equation}
We have reduced the choice of a gauge to the choice of an extension operator $\mathcal{E}$. Note that such an operator can be constructed by first extending with values in $GL(n,\C)$ and then projecting onto $U(n)$ through a strong deformation retract (a continuous map $F: [0,1] \times GL(n,\C) \to GL(n,\C)$ such that $F(0, x) = x$ and $F(1,x) \in U(n)$ for all $x \in GL(n,\C)$, and $F(t,\cdot) \vert_{U(n)} = \Id$ for all $t \in [0,1]$).

In practice, say that we observe the scattering data $S^B$ on past-determined and future-determined paths for some connection $B$ and that we have complete knowledge of the forward map $A \mapsto S^A$. We wish to find the gauge equivalence class of $B$ from $S^B$, which amounts to finding a connection $A$ such that $A = B \triangleleft \varphi$ for some $\varphi \in \mathscr{G}$. Our results give the following strategy to do so.
\begin{enumerate}
	\item By taking $y$ on the boundary of $\mho$, use Theorem \ref{thm:estimatein} to determine $B$ inside $\mho$ from the scattering data of $B$ along past-determined and future-determined paths.
	\item Minimise the mapping
	\begin{equation}
		A \mapsto \norm{S^A_{z_y \ot y \ot x} - S^B_{z_y \ot y \ot x} P^B_{x \ot z_x}}_{H^1(\mathcal{F}_\varepsilon^X)}
	\end{equation}
	over all light-sink connections $A$. Note that we can compute $P^B_{x \ot z_x}$ from the first step since we know $B$ inside $\mho$.
	\item By Corollary \ref{coro:H1} and the definition of $\rho$, the unique minimiser of this problem is $A = \rho([B])$.
	\item Use Proposition \ref{prop:HtoG} to get a connection $A \triangleleft \Phi$ that is gauge-equivalent to $B$.
\end{enumerate}
Note that in step (2), $S^B_{z_y \ot y \ot x} P^B_{x \ot z_x}$ is precisely the scattering data of $\rho([B])$, which explains why Corollary \ref{coro:H1} implies that $A = \rho([B])$ is the unique minimiser of the problem.

One can implement this algorithm with the use of Bayesian inversion. Step (2) is equivalent to recovering a light-sink connection from its scattering data and we will show in Section \ref{sec:stats} that we can consistently do so through Bayesian inversion, see Theorem \ref{thm:stats}. Using similar arguments, one could also provide guarantees for recovering $B$ on $\mho$ in step (1) using Bayesian inversion. As steps (3) and (4) are only simple direct computations, the above algorithm fits within the framework of Bayesian inverse problems. Therefore, by following these steps, one should be able to compute a connection that is close to being gauge-equivalent to $B$ from noisy measurements of its scattering data.

\subsection*{Acknowledgements}

I would like to thank Gabriel Paternain for suggesting this project and for his guidance. I would also like to thank Richard Nickl, Lauri Oksanen and Jan Bohr for their helpful comments. This research was supported by the Cambridge Trust, NSERC's PGS D scholarship and the CCIMI.

\section{Stability estimate}\label{sec:stability}

The goal of this section is to prove the following two pointwise estimates from which Theorems \ref{thm:estimatein} and \ref{thm:estimateout} will follow.

\begin{thm}\label{thm:stabilitymho}
Let $A$ and $B$ be Hermitian connections on $\D$. Then, there is a constant $C > 0$ such that for all $x \in \mho^X$,
\begin{equation}
	\norm{(A-B)_x} \leq C \sup_{\substack{y \in \D\setminus\mho \\ (x,y) \in \mathbb{L}}} \abs{\del_{x \ot y} \left(S^A_{z_y \ot y \ot x} [S^B_{z_y \ot y \ot x}]^{-1}\right)}.
\end{equation}	
\end{thm}

\begin{thm}\label{thm:pointstability}
Let $A$ and $B$ be Hermitian connections on $\D$. There exists a smooth function $p \in C^{\infty}(\D, \C^{n \times n})$ vanishing on $\mathcal{O}$ and $C > 0$ such that for all $0 < \varepsilon < \varepsilon_0$ and $y \in \D\setminus\mho_{\varepsilon}$, it holds that
\begin{equation}\label{eq:pointstability}
	\norm{(A - B - d_{E(A,B)} p)_y} \leq \frac{C}{\varepsilon^4} \int_{(\mathcal{F}^X_\varepsilon)_y} \abs{\del_{y \ot x} \left([S^A_{z_y \ot y \ot x}]^{-1} S^B_{z_y \ot y \ot x}\right)} \de x.
\end{equation}
\end{thm}

To do so, we introduce the attenuated X-ray transform, as well as a pseudolinearisation identity. We also show how to reformulate the theorems in the form of an $H^1$ estimate.

\subsection{The attenuated X-ray transform}

Let $\gamma : [0,T] \to \D$ be a smooth curve and let $\omega \in \Omega^1(\D, \C^n)$, that is, $\omega$ is a one-form on $\D$ with values in $\C^n$ (we will actually use $\C^{n\times n}$ in the proof of Theorem \ref{thm:pointstability}, but everything will be defined analogously through the isomorphism with $\C^{n^2}$). Fix a Hermitian connection $A$ on $\D$ as above. The attenuated X-ray transform of $\omega$ along $\gamma$ with respect to $A$ is given by
\begin{equation}\label{eq:defI}
	I_{\gamma}^A(\omega) := \int_0^T P^A_{\gamma(0) \ot \gamma(t)} \omega(\dot{\gamma}(t)) \de t.
\end{equation}
Similar to the parallel transport, we can express $I^A_\gamma(\omega)$ as the solution of a matrix ODE.

\begin{lem}\label{lem:u(T)}
Let $u$ be the unique solution along $\gamma : [0,T] \to \D$ of the matrix ODE
\begin{equation}
	\begin{cases}
		\dot{u} + A(\dot{\gamma}(t))u = -\omega(\dot{\gamma}(t)), \\
		u(0) = u_0.
	\end{cases}
\end{equation}
Then $u(T) = P^A_\gamma(u_0 - I^A_\gamma(\omega))$.
\end{lem}

\begin{proof}
Let $U$ solve
\begin{equation}
	\begin{cases}
		\dot{U} + A(\dot{\gamma}(t))U = 0, \\
		U(0) = \Id.
	\end{cases}
\end{equation}
A quick computation shows that $\dot{(U^{-1})} = U^{-1}A$. Therefore, along $\gamma$ we have
\begin{equation}
	\dot{(U^{-1} u)} = U^{-1}Au + U^{-1} \dot{u} = U^{-1}(Au + \dot{u}) = -U^{-1}\omega.
\end{equation}
Integrating both sides from $0$ to $T$ yields
\begin{equation}
	U^{-1}(T)u(T) - U^{-1}(0)u(0) = \int_0^T \dot{(U^{-1} u)}(t) \de t = -\int_0^T U^{-1}(t) \omega(\dot{\gamma}(t)) \de t.
\end{equation}
By definition of the parallel transport, $U(t) = P^A_{\gamma(t) \leftarrow \gamma(0)}$. Isolating $u(T)$ in the previous equation and replacing $U$ by the parallel transport yields the result.
\end{proof}

If $A$ vanishes identically, the attenuated X-ray is simply the integral of the one-form $\omega$ along $\gamma$, and so if $\omega$ is potential ($\omega = df$ for some $f \in C^\infty(\D, \C^n)$), then the attenuated X-ray of $\omega$ is the difference between the values of $f$ at both endpoints of $\gamma$ by the fundamental theorem of calculus. This is not exactly true when $A$ does not vanish as we have to account for the parallel transport in the definition of $I^A_\gamma(\omega)$. Instead of potential forms with respect to $d$, we actually have to consider potential forms with respect to $d_A = d + A$ to get an analog of the fundamental theorem of calculus.

\begin{lem}\label{lem:dAf}
Let $f : \D \to \C^n$ be a smooth function on $\D$. Then
\begin{equation}
	I^A_\gamma(d_A f) = (P^A_\gamma)^{-1}f(\gamma(T)) - f(\gamma(0))
\end{equation}
where $d_A f = df + Af$.
\end{lem}

Recall the definition of $\del_{y \ot x}$ as in \eqref{eq:delyx}. We can apply $\del_{y \ot x}$ to the attenuated X-ray to evaluate the values of a one-form from the tangent space at $y$.

\begin{lem}\label{lem:differentiate}
Let $\omega$ be a one-form on $\D$. For $x \neq y$,
\begin{equation}
	\del_{y \leftarrow x} \left(I^A_{y \leftarrow x}(\omega)\right) = P^A_{x \leftarrow y}\left(\omega_y(v_{y \ot x})\right).
\end{equation}
\end{lem}

\begin{proof}
Let $\gamma : [0,T]$ be the line segment from $x$ to $y$ parametrised by arclength. By extending $\gamma$, we see that $\gamma(s) + tv_{y \ot x} = \gamma(s+t)$. Hence, we get
\begin{align*}
	\del_{y \leftarrow x} I^A_{y \leftarrow x}(\omega) &= \der{}{t} \left( I^A_{y + tv \ot x} (\omega)\right)\bigg\vert_{t=0} \\
	&= \der{}{t} \left(\int_{0}^{T + t} P^A_{x \leftarrow \gamma(s)} \omega_{\gamma(s)}(\dot{\gamma}(s)) ds\right)\bigg\vert_{t=0} \\
	&= P^A_{x \leftarrow y} \omega_y(\dot{\gamma}(T))
\end{align*}
since $\gamma(T) = y$. The result follows since $v_{y \ot x} = \dot{\gamma}(T)$.
\end{proof}

\subsection{The broken attenuated X-ray transform}

We will actually be interested in a broken version of the attenuated X-ray transform. One could define naively the broken attenuated X-ray transform $I^A_{z \ot y \ot x}(\omega)$ as $I_{y \ot x}(\omega) + I_{y \ot z}(\omega)$. However, this is not compatible with Lemma \ref{lem:dAf} as we would want 
\begin{equation}\label{eq:brokendAf}
	I^A_{z \ot y \ot x}(d_A f) = P^A_{x \ot y} P^A_{y \ot z} f(z) - f(x)
\end{equation}
to hold in general. It also does not coincide with the usual attenuated X-ray transform $I^A_{z \ot x}(\omega)$ if $x,y$ and $z$ lie on the same line in order. Instead, we need to define the broken attenuated X-ray transform as
\begin{equation}\label{eq:broken}
	I^A_{z \ot y \ot x}(\omega) := I^A_{y \ot x}(\omega) + P^A_{x \ot y} I^A_{z \ot y}(\omega).
\end{equation}
One can check that \eqref{eq:brokendAf} holds under this definition and $I^A_{z \ot y \ot x}(\omega) = I^A_{z \ot x}(\omega)$ whenever the curve $z \ot y \ot x$ is smooth.

\subsection{Pseudolinearisation identity}

The key tool in the proofs of Theorems \ref{thm:stabilitymho} and \ref{thm:pointstability} is the following pseudolinearisation identity. It relates parallel transports along a curve with respect to two different connections with an attenuated X-ray of their difference. See \cite[Chapter 13.2]{GIP2D} for more details on the pseudolinearisation identity. We shall adapt their proof to our setting.

\begin{lem}\label{lem:pseudo}
For any smooth curve $\gamma : [0,T] \to \D$ and connections $A$ and $B$,
\begin{equation}\label{eq:pseudo}
	[P^A_\gamma]^{-1} P^B_\gamma - \Id = I_\gamma^{E(A,B)}(A-B)
\end{equation}
where $E(A,B) \in \mathrm{End}(\C^{n\times n})$ is given by $E(A,B)Q = AQ - QB$ for $Q \in \C^{n\times n}$.
\end{lem}

The right-hand side of \eqref{eq:pseudo} is the attenuated X-ray of $A - B$ with respect to $E(A,B)$. This is slightly different to how we introduced the attenuated X-ray earlier. However, we can see $A - B$ as a one-form taking values in $\C^{n^2} \simeq \C^{n \times n}$ and $E(A,B)$ as a connection on the trivial bundle $\D \times \C^{n^2}$. Before proving Lemma \ref{lem:pseudo}, we state another useful lemma.

\begin{lem}\label{lem:PE(A,B)}
Let $\gamma : [0,T] \to \D$ be a smooth curve and let $A$ and $B$ be connections on $\D$. For any $Q \in \C^{n \times n}$,
\begin{equation}\label{eq:PE(A,B)}
	P^{E(A,B)}_\gamma Q = P^A_\gamma Q [P^B_\gamma]^{-1}.
\end{equation}
\end{lem}

\begin{proof}
Let $u$ and $v$ solve
\begin{equation}
	\begin{cases}
		\dot{u} + E(A,B)(\dot{\gamma}(t))u = 0, \\
		u(0) = Q,
	\end{cases}
\quad \text{ and } \qquad
	\begin{cases}
		\dot{v} + B(\dot{\gamma}(t)) v = 0, \\
		v(0) = \Id,
	\end{cases}
\end{equation}
respectively. On one hand, by the definition of parallel transport, $u(T)v(T) = (P^{E(A,B)}_\gamma Q)P^B_\gamma$. On the other hand,
\begin{equation}
	\dot{(uv)} = [-A(\dot{\gamma}(t)) u + u B(\dot{\gamma}(t))]v - u B(\dot{\gamma}(t)) v = -A(\dot{\gamma}(t)) uv.
\end{equation}
Hence, $uv$ satisfies the parallel transport equation for $A$ along $\gamma$ with $u(0)v(0) = Q$, and so $u(T)v(T) = P^A_\gamma Q$. Combining the two expressions for $u(T)v(T)$ yields \eqref{eq:PE(A,B)}.
\end{proof}

\begin{proof}[Proof of Lemma \ref{lem:pseudo}]
Let $u_A$ solve
\begin{equation}
	\begin{cases}
		\dot{u}_A + A(\dot{\gamma}(t))u_A = 0, \\
		u_A(0) = \Id,
	\end{cases}
\end{equation}
and let $u_B$ be the solution of the same equation with the connection $A$ replaced by $B$. Consider the function $q:= u_Au_B^{-1} - \Id$. From the definition of parallel transport, evaluating $q$ at $T$ yields 
\begin{equation}
q(T) = P^A_{\gamma}[P^B_{\gamma}]^{-1} - \Id	.
\end{equation}
Moreover, one can check that $q$ solves
\begin{equation}
	\begin{cases}
		\dot{q} + E(A,B)(\dot{\gamma}(t)) q = -(A-B)(\dot{\gamma}(t)), \\
		q(0) = 0.
	\end{cases}
\end{equation}
Lemma \ref{lem:u(T)} then yields $q(T) = -P^{E(A,B)}_\gamma I^{E(A,B)}_\gamma(A-B)$. By combining the expressions for $q(T)$ and applying Lemma \ref{lem:PE(A,B)}, we get
\begin{equation}
	P^A_{\gamma}[P^B_\gamma]^{-1} - \Id = -P^A_\gamma I^{E(A,B)}_\gamma(A-B) [P^B_\gamma]^{-1}.
\end{equation}
Rearranging the last equation yields \eqref{eq:pseudo}.
\end{proof}

Importantly, the pseudolinearisation identity is also valid in the broken case, where the parallel transports are replaced by the scattering data.

\begin{lem}\label{lem:brokenpseudo}
Let $A$ and $B$ be connections on $\D$ and let $x,y,z \in \D$. Then
\begin{equation}\label{eq:pseudobroken}
	[S^A_{z \ot y \ot x}]^{-1}S^B_{z \ot y \ot x} - \Id = I^{E(A,B)}_{z \ot y \ot x} (A - B).
\end{equation}
\end{lem}

\begin{proof}
By expanding $I^{E(A,B)}_{z \ot y \ot x} (A-B)$, we get
\begin{equation}
	I^{E(A,B)}_{z \ot y \ot x} (A-B) = I^{E(A,B)}_{y \ot x} (A-B) + P^{E(A,B)}_{x \ot y}I^{E(A,B)}_{z \ot y} (A-B).
\end{equation}
We can use Lemma \ref{lem:pseudo} on both attenuated X-ray transforms and Lemma \ref{lem:PE(A,B)} on the parallel transport to get
\begin{align*}
	I^{E(A,B)}_{z \ot y \ot x} (A-B) &= P^A_{x \ot y}P^B_{y \ot x} - \Id + P^A_{x \ot y}(P^A_{y \ot z} P^B_{z \ot y} - \Id)P^B_{y \ot x} \\
	&= P^A_{x \ot y}P^B_{y \ot x} - \Id + P^A_{x \ot y} P^A_{y \ot z} P^B_{z \ot y} P^B_{y \ot x} - P^A_{x \ot y}P^B_{y \ot x} \\
	&= [S^A_{z \ot y \ot x}]^{-1} S^B_{z \ot y \ot x} - \Id
\end{align*}
as claimed.
\end{proof}

The pseudolinearisation identity and Lemma \ref{lem:differentiate} are enough to prove Theorem \ref{thm:stabilitymho}.

\begin{proof}[Proof of Theorem \ref{thm:stabilitymho}]
By interchanging the role of $x$ and $z$, we see that the pseudolinearisation identity can also be written as
\begin{equation}
	S^A_{z \ot y \ot x} [S^B_{z \ot y \ot x}]^{-1} - \Id = I^{E(A,B)}_{x \ot y \ot z}(A-B).
\end{equation}
Hence, by definition of the broken X-ray, we have
\begin{equation}
	\del_{x \ot y}\left(S^A_{z \ot y \ot x} [S^B_{z \ot y \ot x}]^{-1}\right) = \del_{x \ot y} \left(I^{E(A,B)}_{y \ot z}(A-B) + P^{E(A,B)}_{z \ot y}I^{E(A,B)}_{x \ot y}(A-B) \right).
\end{equation}
The operator $\del_{x \ot y}$ is essentially a derivative with respect to $x$, and so the first term in the definition of the broken X-ray vanishes. Moreover, $P^{E(A,B)}_{z \ot y}$ is unaffected. It follows from Lemmas \ref{lem:differentiate} and \ref{lem:PE(A,B)} that
\begin{align}
	\del_{x \ot y}\left(S^A_{z \ot y \ot x} [S^B_{z \ot y \ot x}]^{-1}\right) &= P^{E(A,B)}_{z \ot y} \del_{x \ot y}\left(I_{x \ot y}^{E(A,B)}(A-B)\right) \\
	&= P^{E(A,B)}_{z \ot y} P^{E(A,B)}_{y \ot x} (A-B)_x(v_{x \ot y}) \\
	&= S^A_{z \ot y \ot x}[(A-B)_x(v_{x \ot y})] S^B_{z \ot y \ot x}.
\end{align}
Taking norms, the scattering data vanish since they belong in $U(n)$ and we get
\begin{equation}\label{eq:pseudomho}
	\abs{(A-B)_x(v_{x \ot y})} = \abs{\del_{x \ot y}\left(S^A_{z \ot y \ot x} [S^B_{z \ot y \ot x}]^{-1}\right)}.
\end{equation}
The choice of $z$ on the right-hand side is irrelevant, and we take $z = z_y$. Since vectors of the form $v_{x \ot y}$ form a basis of the tangent plane at $x$ without degenerating when $\varepsilon$ goes to $0$, we can find a constant $C > 0$ such that
\begin{equation}
	\sup_{\substack{v \in T_{x}\D \\ \abs{v} = 1}} \abs{(A-B)_x(v)} \leq C \sup_{\substack{y \in \D \setminus \mho \\ (x,y) \in \mathbb{L}}} \abs{(A-B)_x(v_{x \ot y})} = C \sup_{\substack{y \in \D \setminus \mho \\ (x,y) \in \mathbb{L}}} \abs{\del_{x \ot y}\left(S^A_{z \ot y \ot x} [S^B_{z \ot y \ot x}]^{-1}\right)}
\end{equation}
and the theorem follows.
\end{proof}

To prove Theorem \ref{thm:estimatein}, it only remains to integrate over $\mho^X$ to get a global estimate.

\begin{proof}[Proof of Theorem \ref{thm:estimatein}]
By equivalence of norms, we can find $C > 0$ independent of both $\varepsilon$ and $x \in \mho^X$ such that
\begin{equation}
	\norm{(A-B)_x} \leq C \int_{y \, : \, x \in \mathcal{F}^X_y} \abs{(A-B)_x(v_{x \ot y})} \de y.
\end{equation}
After changing the integrand through equation \eqref{eq:pseudomho} with $z = z_y$, integrating over $x \in \mho^X$ and using Cauchy-Schwarz yields the desired estimate.
\end{proof}

The proof of Theorem \ref{thm:stabilitymho} crucially relies on the fact that $x$ is always an endpoint of the path and is not the breaking point, since then the operator $\del_{x \ot y}$ only hits $I^{E(A,B)}_{x \ot y}$ in the expression for the broken attenuated X-ray. This allows us to evaluate $A-B$ inside $\mho$, but such an approach does not immediately work for evaluating $A-B$ outside $\mho$. This is where we need to take the gauge into account.

\subsection{Dealing with the gauge through a potential form}
In order to use similar techniques as in the proof of Theorem \ref{thm:stabilitymho} to estimate the connection outside $\mho$, we aim to make the second term in \eqref{eq:broken} vanish. To do so, we will modify the argument of the attenuated X-ray by a potential form.

For a connection $A$ and a one-form $\omega$, we define the function
\begin{equation}\label{eq:p}
	p(y) := p^A_\omega(y) = P^A_{y \ot z_y} I^A_{y \ot z_y}(\omega).
\end{equation}
This function will serve as an approximate potential for $\omega$. We chose $p$ in this way so that $I^A_{z_y \ot y}(\omega - d_A p)$ vanishes for all $y \in \D \setminus \mho$, as the next lemma shows.

\begin{lem}\label{lem:gammaout}
Let $\gamma :[0, T] \to \D$ be the unit-speed lightlike geodesic from $y$ to $z_y$. Then, with $p$ defined as above, we have
\begin{equation}
	\omega(\dot{\gamma}(t)) = d_A p(\dot{\gamma}(t))
\end{equation}
for all $t \in (0,T)$. In particular, $\omega(v_{y \ot z_y}) = d_A p(v_{y \ot z_y})$.
\end{lem}

\begin{proof}
Consider the unique solution $u$ of
\begin{equation}
	\begin{cases}
		\dot{u} + A(\dot{\gamma}(t)) u = -\omega(\dot{\gamma}(t)), \\
		u(0) = 0,
	\end{cases}
\end{equation}
along $\gamma$. By Lemma \ref{lem:u(T)}, it holds that $u(t) = -P^A_{\gamma(t) \ot z_y} I^A_{\gamma(t) \ot z_y} = -p(\gamma(t))$. Hence, we have
\begin{equation}
	-d_A p(\dot{\gamma}(t)) = \dot{u}(t) + A(\dot{\gamma}(t))u(t) = -\omega(\dot{\gamma}(t))
\end{equation}
and the result follows.
\end{proof}

We can deduce from Lemma \ref{lem:gammaout} and \eqref{eq:defI} that $I^A_{z_y \ot y}(\omega - d_A p) = 0$ for all $y$ and so, on the one hand,
\begin{equation}
	I^A_{z_y \ot y \ot x}(\omega - d_A p) = I^A_{y \ot x}(\omega - d_A p).
\end{equation}
On the other hand, by \eqref{eq:brokendAf}, we have
\begin{align*}
	I^A_{z_y \ot y \ot x}(\omega - d_A p) &= I^A_{z_y \ot y \ot x}(\omega) - P^A_{x \ot y}P^A_{y \ot z} p(z_y) + p(x) \\
	&= I^A_{z_y \ot y \ot x}(\omega) + p(x)
\end{align*}
since $p(z_y) = 0$ and so by combining both expressions, we get
\begin{equation}\label{eq:Ialpha}
	I^A_{z_y \ot y \ot x}(\omega) = I^A_{y \ot x}(\omega - d_A p) - p(x).
\end{equation}
By applying $\del_{y \ot x}$ to both sides of the last expression, Lemma \ref{lem:differentiate} yields
\begin{equation}\label{eq:delyxbroken}
	\del_{y \ot x} \left(I^A_{z_y \ot y \ot x}(\omega)\right) = P^A_{x \ot y}\left([\omega - d_A p]_y(v_{y \ot x})\right)
\end{equation}
since $\del_{y \ot x}$ is essentially a derivative in $y$, and so $\del_{y \ot x} p(x) = 0$. To prove Theorem \ref{thm:pointstability}, we will replace $A$ by $E(A,B)$ and $\omega$ by $A-B$ in \eqref{eq:delyxbroken} in order to use the pseudolinearisation identity.

\subsection{Evaluating from the tangent space at $y$}

As shown in \cite[Lemma 1]{chen2019detection} , the set of vectors $v_{y \ot x}$ for $x \in (\mathcal{F}_\varepsilon^X)_y$ form a basis of the tangent space at $y$, but this basis degenerates when $\varepsilon$ goes to $0$. We therefore need estimates to quantify how well we can estimate $\omega - d_A p$ at $y \in \D\setminus\mho$ from moving $x$ around in the intersection of $\mho$ and the past lightcone of $y$.

\begin{lem}\label{lem:sample}
Let $0 < \varepsilon < \varepsilon_0$ and let $y \in \D \setminus \mho_{\varepsilon}$. Then
\begin{equation}
	\norm{(\omega - d_A p)_y}\leq \frac{8}{\eps}\sup_{\substack{x \in (\mathcal{F}^X_\varepsilon)_y}} \abs{(\omega - d_A p)(v_{y \ot x})}.
\end{equation}
\end{lem}

The key to proving Lemma \ref{lem:sample} is this small linear algebra lemma whose proof is straightforward.

\begin{lem}\label{lem:linalg}
Let $b_1, \dots, b_m$ be a basis of $\R^m$ with $\abs{b_i}_{\R^m} = 1$ and let $\mathcal{T}: \R^m \to \C^n$ be a linear map. Then
\begin{equation}
	\norm{\mathcal{T}} \leq \sqrt{m}\norm{B^{-1}} \max_{i = 1, \dots, m} \abs{\mathcal{T}(b_i)}
\end{equation}
where $B$ is the matrix whose columns are the $b_i$'s and $\norm{B^{-1}}$ is the operator norm of its inverse.
\end{lem}

\begin{proof}[Proof of Lemma \ref{lem:sample}]
As stated earlier, Lemma 1 in \cite{chen2019detection} guarantees that the set of vectors $v_{y \ot x}$ generate $T_y \R^{1+3}$. Hence, we wish to apply Lemma \ref{lem:linalg} by evaluating from $T_y \R^{1+3}$ using different light rays $\gamma_x$ from $x$ to $y$ for different $x \in \mho_\eps$ with $(x,y) \in \mathbb{L}$, that is, $x$ in the fibre of $\mathcal{F}_\varepsilon^X$ at $y$. 

We first claim that it suffices to compute the case where $y = (0,1,0,0)$. Through a rotation in space and a translation in time, we can identify the sets $\{v_{y \ot x}\}_{x \in \mho_{\varepsilon}}$ and $\{v_{y' \ot x}\}_{x \in \mho_{\varepsilon}}$ whenever $y$ and $y'$ share the same spatial norm. By symmetry, this does not intervene in norm estimates. Therefore, without loss of generality, we can choose $y = y_r = (0,r,0,0)$. Moreover, whenever $r_1 < r_2$, we can see that $\{v_{y_{r_2} \ot x}\}_{x \in \mho_{\varepsilon}} \subset \{v_{y_{r_1} \ot x}\}_{x \in \mho_\varepsilon}$ and so any stability estimate for $y_{r_2}$ is also valid for $y_{r_1}$ since we're taking the supremum over a larger set. Hence, it suffices to show the case $r=1$, as claimed.

To apply Lemma \ref{lem:linalg}, we need a basis of $T_y \R^{1+3}$. Let $b_i := v_i/\abs{v_i}$ where
\begin{align*}
	v_1 &= (1, 1, 0, 0),\\
	v_2 &= (\sqrt{1 + \varepsilon^2}, 1, \varepsilon, 0),\\
	v_3 &= (\sqrt{1 + \varepsilon^2}, 1, 0, \varepsilon),\\
	v_4 &= (-1, 1, 0, 0).
\end{align*}
It is obvious that the $b_i$'s are linearly independent and hence form a basis of the tangent space at $y$. The vector $b_4$ is $v_{y \ot z_y}$, while the other vectors $b_i$ correspond to $v_{y \ot x_i}$ with $x_1 = (-1,0,0,0)$, $x_2 = (-\sqrt{1 + \varepsilon^2}, 0, -\varepsilon,0)$ and $x_3 = (-\sqrt{1 + \varepsilon^2}, 0, 0, -\varepsilon)$. Notice that $(x_i, y) \in \mathbb{L}$ and that $x_i \in \overline{\mho_\eps}$ for $i = 1,2,3$. A quick computation with Mathematica yields
\begin{equation}
	\norm{B^{-1}}_{\mathrm{F}} = \frac{\sqrt{8 \varepsilon^2 + 8}}{\varepsilon} \leq \frac{4}{\varepsilon}
\end{equation}
for $0 < \varepsilon < 1$. The operator norm and the Frobenius norm are equivalent with $\norm{B^{-1}} \leq \norm{B^{-1}}_{\mathrm{F}}$ and so by Lemma \ref{lem:linalg},
\begin{align*}
	\norm{(\omega - d_A p)_y} &= \sup_{v \in T_y \R^{1+3}} \frac{\abs{(\omega - d_A p)(v)}}{\abs{v}} \\
	&\leq 2 \norm{B^{-1}} \max_{i = 1, 2, 3, 4} \abs{(\omega - d_A p)(b_i)} \\
	&\leq \frac{8}{\eps} \sup_{\substack{x \in \mho_\eps \\ (x,y) \in \mathbb{L}}} \abs{(\omega - d_Ap)(v_{y \ot x})}.
\end{align*}
The last inequality follows from the fact that $(\omega - d_A p)(b_4) = 0$ by Lemma \ref{lem:gammaout} and since $\{b_1,b_2,b_3\}$ is in the closure of $\{v_{y \ot x}\}_{x \in \mho_\varepsilon}$.
\end{proof}

\subsection{Proof of Theorems \ref{thm:estimateout} and \ref{thm:pointstability}}

We finally have everything to prove Theorem \ref{thm:pointstability}. The main idea is to use the pseudolinearisation identity to relate the scattering data with an attenuated X-ray transform of $A - B$, and then use the operator $\del_{y \ot x}$ to evaluate $A - B - d_{E(A,B)} p$ from $T_y \R^{1+3}$. Theorem \ref{thm:estimateout} then immediately follows by integrating over $\D \setminus \mho$.

\begin{proof}[Proof of Theorem \ref{thm:pointstability}]
By Lemma \ref{lem:brokenpseudo}, we have
\begin{equation}
	[S^A_{z_y \ot y \ot x}]^{-1}S^B_{z_y \ot y \ot x} - \Id = I^{E(A,B)}_{z_y \ot y \ot x} (A - B).
\end{equation}
We can see $A$ and $B$ as one-forms taking values in $\C^{n^2}$ and $E(A,B)$ a Hermitian connection taking values in $\mathfrak{u}(n^2)$. Hence, if we let
\begin{equation}
	p = p_{A-B}^{E(A,B)} = P^{E(A,B)}_{y \ot z_y} I^{E(A,B)}_{y \ot z_y}(A - B)
\end{equation}
then \eqref{eq:Ialpha} yields
\begin{equation}
	I^{E(A,B)}_{z_y \ot y \ot x}(A - B) = I^{E(A,B)}_{y \ot x}(A - B - d_{E(A,B)} p) - p(x).
\end{equation}
Since $\del_{y \ot x} \Id = \del_{y \ot x} p(x) = 0$, it now follows from Lemma \ref{lem:differentiate} that
\begin{align*}
	\del_{y \ot x}\left([S^A_{z_y \ot y \ot x}]^{-1}S^B_{z_y \ot y \ot x}\right) &= P^{E(A,B)}_{y \ot x}(A - B - d_{E(A,B)} p)(v_{y \ot x}) \\
	&= P^{A}_{y \ot x}(A - B - d_{E(A,B)} p)(v_{y \ot x}) P^{B}_{x \ot y}.
\end{align*}
The parallel transports are in $U(n)$ and so
\begin{equation}\label{eq:dellinear}
	\abs{\del_{y \ot x}\left([S^A_{z_y \ot y \ot x}]^{-1}S^B_{z_y \ot y \ot x}\right)} = \abs{(A - B - d_{E(A,B)} p)(v_{y \ot x})}.
\end{equation}
We can finally apply Lemma \ref{lem:sample} to get
\begin{equation}
	\norm{(A - B - d_{E(A,B)} p)_y} \leq \frac{8}{\eps} \sup_{x \in (\mathcal{F}_\varepsilon^X)_y} \abs{(A - B - d_{E(A,B)} p)(v_{y \ot x})}.
\end{equation}
Finally, note that
\begin{equation}
	\sup_{x \in (\mathcal{F}^X_{\varepsilon})_y} \abs{\omega(v_{y \ot x})} \sim \frac{1}{\vol((\mathcal{F}^X_{\varepsilon})_y)} \int_{(\mathcal{F}^X_\varepsilon)_y} \abs{\omega(v_{y \ot x})} \de x
\end{equation}
for any one-form $\omega$ as $\varepsilon$ goes to $0$. Combining this with the fact that $\vol((\mathcal{F}^X_{\varepsilon})_y)$ is proportional to $\varepsilon^3$, we can find a constant $C > 0$ independent of $\varepsilon$ such that
\begin{equation}
	\norm{(A-B-d_{E(A,B)}p)_y} \leq \frac{C}{\varepsilon^4} \int_{(\mathcal{F}^X_\varepsilon)_y} \abs{\del_{y \ot x}\left([S^A_{z_y \ot y \ot x}]^{-1}S^B_{z_y \ot y \ot x}\right)} \de x.
\end{equation}
\end{proof}

\begin{rem}
Note that even though there is a supremum in the right-hand side of \eqref{eq:pointstability}, one does not need to know $\del_{y \ot x}\left([S^A_{z_y \ot y \ot x}]^{-1}S^B_{z_y \ot y \ot x}\right)$ for all $x \in \mho_\varepsilon$ in the past light cone of $y$ to get an estimate. Indeed, the important equation is \eqref{eq:dellinear} as it reveals the linear structure behind the estimate. In practice, one only needs to evaluate $A - B - d_{E(A,B)}p$ at three different linearly independent vectors $v_{y \ot x}$ since we already know it vanishes when evaluated at $v_{y \ot z_y}$. Lemma \ref{lem:linalg} then yields an estimate for those vectors.
\end{rem}

\subsection{$H^1$ estimate}

It remains to show Corollary \ref{coro:H1}, which relates $S^A$ and $S^B$ in a linear fashion rather than through the group multiplication in $U(n)$. To do so, we follow the argument in \cite[Corollary 2.3]{monard2021consistent}.

\begin{lem} \label{lem:linear}
There is a constant $C > 0$ such that
\begin{equation}
	\norm{\del([S^A]^{-1}S^B)}_{L^2(\mathcal{F}^X)} \leq C\left(1 + \Vert A \triangleleft P^A_{y \ot z_y} \Vert_{L^\infty(\D \setminus \mho)}\right)\norm{S^A - S^B}_{H^1(\mathcal{F}^X)}.
\end{equation}
\end{lem}

\begin{proof}
To simplify notation, we omit the paths in what follows and write $S^A$ for $S^A_{z_y \ot y \ot x}$. We can expand
\begin{align*}
	\abs{\del_{y \ot x}([S^A]^{-1} S^B)} &= \abs{(\del_{y \ot x} [S^A]^{-1})S^B + [S^A]^{-1} \del_{y \ot x}S^B} \\
	&= \abs{[S^A]^{-1} \del_{y \ot x}S^B - [S^A]^{-1} (\del_{y \ot x} S^A) [S^A]^{-1} S^B} \\
	&= \abs{\del_{y \ot x}(S^B - S^A) + (\del_{y \ot x}S^A) (\Id - [S^A]^{-1}S^B)} \\
	&\leq \abs{\del_{y \ot x}(S^B - S^A)} + \abs{\del_{y \ot x}S^A} \abs{S^A - S^B}.
\end{align*}
The third equality follows from the second by using that $S^A \in U(n)$ as well as adding and substracting $\del_{y \ot x} S^A$. Taking the supremum over the fibres $\mathcal{F}^X_y$, squaring, integrating and using that $(a+b)^2 \leq 2(a^2 + b^2)$ yields
\begin{equation}
	\norm{\del([S^A]^{-1} S^B)}_{L^2(\mathcal{F}^X)}^2 \leq 2\left(\norm{\del(S^A - S^B)}_{L^2(\mathcal{F}^X)} ^2 + \norm{\del S^A}_{L^\infty(\mathcal{F}^X)}^2 \norm{S^A - S^B}_{L^2(\mathcal{F}^X)}^2\right).
\end{equation}

It remains to estimate $\norm{\del S^A}_{L^\infty(\mathcal{F}^X)}$. We did not show it yet, but the proof of Theorem \ref{thm:gauge} reveals that
\begin{equation}
	\abs{\del_{y \ot x} S^A_{z_y \ot y \ot x}} = \abs{(A \triangleleft P^A_{y \ot z_y})(v_{y \ot x})}
\end{equation}
and so $\norm{\del S^A}_{L^\infty(\mathcal{F}^X)} \leq \norm{A \triangleleft P^A_{y \ot z_y}}_{L^\infty(\D \setminus \mho)}$. The estimate follows by taking square roots and using that $\sqrt{1 + x^2}/(1+x)$ is bounded.
\end{proof}

The last estimate is again invariant under $\mathscr{G}$ and involves the $L^{\infty}$-norm of the light-sink connection $A \triangleleft P^A_{y \ot z_y}$. We can get an estimate on that norm involving the curvature of $A$ and the value of $A$ along $\mathcal{O}$.

\begin{lem}\label{lem:lightsinkcurvature}
There is a constant $C$	such that
\begin{equation}
	\norm{A \triangleleft P^A_{y \ot z_y}}_{L^\infty(\D)} \leq C\left(\norm{F_A}_{L^\infty(\D)} + \norm{A(\del_t)}_{L^\infty(\mathcal{O})}\right)
\end{equation}
where $F_A = dA + A \wedge A$ is the curvature $2$-form of $A$ and 
\begin{equation}
	\norm{F_A}_{L^\infty(\mathcal{\D})} = \sup_{\substack{y \in \D\\ \abs{u} = \abs{v} = 1}} \abs{(F_A)_y(u,v)}.
\end{equation}
\end{lem}

Corollary \ref{coro:H1} will then directly follow from Theorem \ref{thm:estimateout}, Theorem \ref{thm:gauge}, Lemma \ref{lem:linear} and Lemma \ref{lem:lightsinkcurvature}. However, we need another lemma before proving Lemma \ref{lem:lightsinkcurvature}.

\begin{lem}\label{lem:variation}
Let $\gamma : [0,T] \to \D$ be a smooth curve and let $\gamma_s : [0,T] \to \D$ be a smooth variation of $\gamma$ where $s \in I = [-\delta, \delta]$ for some $\delta > 0$. Then
\begin{align}\label{eq:variation}
	\der{}{s} P^A_{\gamma_s}\bigg\vert_{s = 0} &= P^A_{\gamma} A_{\gamma(0)} (\del_s \gamma_s(0)) - A_{\gamma(T)}(\del_s \gamma_s(T)) P^A_{\gamma} \\
	&\qquad + \int_{0}^T P^A_{\gamma[t,T]}F_A(\dot{\gamma}(t), \del_s \gamma_s(t))P^A_{\gamma[0,t]} \de t
\end{align}
where $\del_s \gamma_s(t) = \der{}{s} \gamma_s(t) \vert_{s=0}$ and $P^A_{\gamma[0,t]}$ is the parallel transport along the segment of $\gamma$ restricted to the interval $[0,t]$.
\end{lem}

\begin{proof}
Let $U(s,t) = P^A_{\gamma_s[0,t]}$. Then $U$ solves
\begin{equation}
\begin{cases}
	\del_t U(s,t) + A(\dot{\gamma}_s(t)) U(s,t) = 0, & (s,t) \in I \times [0,T]; \\
	U(s,0) = \Id.
\end{cases}
\end{equation}
By differentiating with respect to $s$, we get
\begin{equation}
\begin{cases}
	\del_s \del_t U(s,t) + \del_s\big[A(\dot{\gamma}_s(t))\big] U(s,t) + A(\dot{\gamma}_s(t)) \del_s U(s,t) = 0, & (s,t) \in I \times [0,T];\\
	\del_s U(s,0) = 0.
\end{cases}
\end{equation}
Let $v(s,t) = \del_s U(s,t)$. We are interested in computing $v(0,T)$. From the previous equation, we see that $v$ satisfies the inhomogeneous differential equation
\begin{equation}
\begin{cases}
	\del_t v + A(\dot{\gamma}_s(t)) v = -\del_s \big[A(\dot{\gamma}_s(t))\big] P^A_{\gamma_s[0,t]}, & (s,t) \in I \times [0,T]; \\
	v(s,0) = 0.
\end{cases}
\end{equation}
By Duhamel's principle, the solution of this differential equation is given by
\begin{equation}
	v(s,t) = \int_0^t u^r(s,t) \de r
\end{equation}
where $u^r$ solves
\begin{equation}
\begin{cases}
	\del_t u^r + A(\dot{\gamma}_s(t)) u^r = 0, & (s,t) \in I \times (r,T); \\
	u^r(s,r) = -\del_s \big[A(\dot{\gamma}_s(r))\big] P^A_{\gamma_s[0,r]}.
\end{cases}
\end{equation}
The equation defining $u^r$ is simply that of a parallel transport and so
\begin{equation}
	u^r(s,t) = P^A_{\gamma_s[r,t]} \left(-\del_s\big[A(\dot{\gamma}_s(r))\big]\right) P^A_{\gamma_s[0,r]}.
\end{equation}
Hence, we get
\begin{equation}
	v(0,T) = \int_0^T P^A_{\gamma[r,T]} \left(-\del_s\big[A(\dot{\gamma}_s(r))\big] \big\vert_{s = 0}\right) P^A_{\gamma[0,r]} \de r.
\end{equation}
Expanding the curvature $2$-form $F_A = dA + A \wedge A$ yields
\begin{align*}
F_A(\dot{\gamma}(r), \del_s \gamma_s(r)) &= \del_r\big[A(\del_s\gamma_s(r)\big] - \del_s\big[A(\dot{\gamma}_s(r))\big] - A([\dot{\gamma}(r), \del_s\gamma_s(r)]) \\
& \qquad + A(\dot{\gamma}(r))A(\del_s\gamma_s(r)) - A(\del_s \gamma_s(r))A(\dot{\gamma}(r)).
\end{align*}
The vectors $\dot{\gamma}(r)$ and $\del_s \gamma_s(r)$ commute so the term with the commutator vanishes. We can isolate $-\del_s \big[A(\dot{\gamma}_s(r))\big]$ in that expression to get
\begin{align}\label{eq:v0T}
v(0,T) &= \int_0^T P^A_{\gamma[r,T]} \bigg[F_A(\dot{\gamma}(r), \del_s \gamma_s(r)) - \del_r\big[A(\del_s\gamma_s(r))\big] - A(\dot{\gamma}(r))A(\del_s\gamma_s(r)) \\
& \qquad + A(\del_s \gamma_s(r))A(\dot{\gamma}(r)) \bigg] P^A_{\gamma[0,r]} \de r.
\end{align}
We can integrate by parts the last term using that $A(\dot{\gamma}(r)) P^A_{\gamma[0,r]} = -\del_r P^A_{\gamma[0,r]}$. This yields
\begin{align*}
	\int_0^T P^A_{\gamma[r,T]} A(\del_s\gamma_s(r)) A(\dot{\gamma}(r)) P^A_{\gamma[0,r]} \de r &= \bigg[-P^A_{\gamma[r,T]} A(\del_s\gamma_s(r)) P^A_{\gamma[0,r]}\bigg]_0^T \\ &\qquad + \int_0^T \del_r\big[P^A_{\gamma[r,T]} A(\del_s\gamma_s(r))\big] P^A_{\gamma[0,r]} \de r.
\end{align*}
The boundary term corresponds to the first two terms in \eqref{eq:variation} and we can expand the integrand of the second term to get
\begin{equation}
	\del_r\big[P^A_{\gamma[r,T]} A(\del_s\gamma_s(r))\big] P^A_{\gamma[0,r]} = P^A_{\gamma[r,T]} \bigg[ A(\dot{\gamma}(r))A(\del_s\gamma_s(r)) + \del_r \big[A(\del_s\gamma_s(r))\big]\bigg]P^A_{\gamma[0,r]}
\end{equation}
since $\del_r P^A_{\gamma[r,T]} = P^A_{\gamma[r,T]} A(\dot{\gamma}(r))$. These terms cancel with the second and third terms in \eqref{eq:v0T} to simplify $v(0,T)$ to \eqref{eq:variation}.
\end{proof}

\begin{proof}[Proof of Lemma \ref{lem:lightsinkcurvature}]
For $y \in \D$ and unit $v \in T_y \D$, we have
\begin{equation}
	(A \triangleleft P^A_{y \ot z_y})_y (v) = P^A_{z_y \ot y} dP^A_{y \ot z_y} (v) + P^A_{z_y \ot y} A_y(v) P^A_{y \ot z_y}.
\end{equation}
We can compute $dP^A_{y \ot z_y}(v)$ by using Lemma \ref{lem:variation} with the variation $\gamma_s$ given by the lightlike geodesic from $z_{y + sv}$ to $y + sv$. This yields
\begin{align*}
	dP^A_{y \ot z_y}(v) &= P^A_{y \ot z_y} A_{z_y}(c\del_t) - A_y(v) P^A_{y \ot z_y} \\ 
	& \qquad+ \int_0^T P^A_{y \ot \gamma(t)} F_A(\dot{\gamma}(t), \del_s\gamma_s(t)) P^A_{\gamma(t) \ot z_y} \de t.
\end{align*}
for some $-1 \leq c \leq 1$ and hence
\begin{equation}
	(A \triangleleft P^A_{y \ot z_y})(v) = cA_{z_y}\left(\del_t\right) + \int_0^T P^A_{y \ot \gamma(t)} F_A(\dot{\gamma}(t), \del_s\gamma_s(t)) P^A_{\gamma(t) \ot z_y} \de t.
\end{equation}
The result follows since $z_y \in \mathcal{O}$ for all $y$ and
\begin{equation}
	\abs{\int_0^T P^A_{y \ot \gamma(t)} F_A(\dot{\gamma}(t), \del_s\gamma_s(t)) P^A_{\gamma(t) \ot z_y} \de t} \leq T \norm{F_A}_{L^\infty(\D)}.
\end{equation}
\end{proof}

\subsection{Forward estimates}
We finish the section by collecting forward estimates that will be useful for Section \ref{sec:stats}.

\begin{lem}\label{lem:forward}
Let $A$ and $B$ be Hermitian connections. Then
\begin{equation}
	\norm{S^A_{z \ot y \ot x} - S^B_{z \ot y \ot x}}_{L^\infty(\mathbb{S}^+(\mho))} \leq 2\sqrt{2} \norm{A - B}_{L^\infty(S\D)}
\end{equation}
where $S\D = \{(x,v) \in T\D : \abs{v}_e = 1\}$ is the (Euclidean) sphere bundle on $\D$.
\end{lem}

\begin{proof}
	Since $S^A_{z \ot y \ot x}$ and the parallel transports lie in $U(n)$, Lemma \ref{lem:brokenpseudo} yields the pointwise estimate
\begin{align*}
	\abs{S^A_{z \ot y \ot x} - S^B_{z \ot y \ot x}} &= \abs{I^{E(A,B)}_{z \ot y \ot x}(A-B)} \\
	&\leq \abs{I^{E(A,B)}_{y \ot x}(A - B)} + \abs{I_{z \ot y}^{E(A,B)}(A-B)} \\
	&\leq \int_{0}^{T_1} \abs{(A-B)(\dot{\gamma}_{y \ot x}(t))} \de t + \int_{0}^{T_2} \abs{(A-B)(\dot{\gamma}_{z \ot y}(t))} \de t \\
	&\leq \abs{x - y}_e \sup_{t \in [0,T_1]} \abs{(A-B)(\dot{\gamma}_{y \ot x}(t))} + 
	\\ & \qquad + \abs{y - z}_e \sup_{t \in [0,T_2]} \abs{(A-B)(\dot{\gamma}_{z \ot y}(t))}
\end{align*}
where $\gamma_{y \ot x} : [0,T_1] \to \D$ and $\gamma_{z \ot y} : [0,T_2] \to \D$ are the unit speed geodesics from $x$ to $y$ and $y$ to $z$, respectively, and $\abs{x - y}_e$ is the Euclidean distance between $x$ and $y$. In particular, by taking the supremum over all values of $(x,y,z) \in \mathbb{S}^+(\mho)$ and using that the distance between $x$ and $y$ is at most $\sqrt{2}$, we get
\begin{align*}
	\norm{S^A_{z \ot y \ot x} - S^B_{z \ot y \ot x}}_{L^\infty(\mathbb{S}^+(\mho))} &\leq 2\sqrt{2} \sup_{\substack{x, z \in \mho, y \in \D \\ (x,y), (y,z) \in \mathbb{L}}} \left\{\abs{(A-B)(v_{y \ot x})}, \abs{(A-B)(v_{y \ot z})} \right\} \\
	&\leq 2\sqrt{2} \norm{A-B}_{L^\infty(S\D)}
\end{align*}
since the second supremum is taken over a larger set. Note that we did not restrict $y$ outside $\mho$ in the first supremum since the curves $\gamma_{y \ot x}$ and $\gamma_{y \ot z}$ cross $\mho$.
\end{proof}

\begin{lem}\label{lem:forwardattenuated}
Let $A$ be a Hermitian connection and $\omega$ a one-form on $\D$. There is a constant $C > 0$ such that
\begin{equation}
	\norm{I^A_{y \ot x}(\omega)}_{H^k(\D \times \D)} \leq C \norm{P^A_{y \ot x}}_{C^k(\D \times \D)} \norm{\omega}_{H^k(S\D)}
\end{equation}
for all $k \geq 0$.
\end{lem}

\begin{proof}
Let $I : C^\infty(\D \times S\D) \to C^\infty(\D \times \D)$ be the usual ray transform on $\D \times S\D$ given by
\begin{equation}
	I(x,y)(\beta) = \int_0^{T(x,y)} \beta(x, \gamma_{y \ot x}(t), \dot{\gamma}_{y \ot x}(t)) \de t.
\end{equation}
for $\beta \in C^\infty(\D \times S\D)$. It is well-known that $I$ is continuous from $H^k(\D \times S\D)$ to $H^k(\D \times \D)$ for all $k \geq 0$ \cite[Theorem 4.2.1]{sharafutdinov2012integral}. Consider the natural projections in the following diagram.
\begin{center}
\begin{tikzcd}
S\D & \D \times S\D \arrow[l, "\pi_1"'] \arrow[r, "\pi_2"] & \D \times \D.
\end{tikzcd}
\end{center}
These projections induce pullbacks on functions.
\begin{center}
\begin{tikzcd}
C^\infty(S\D) \arrow[r, "\pi_1^*"] & C^\infty(\D \times S\D) & C^\infty(\D \times \D) \arrow[l, "\pi_2^*"'].
\end{tikzcd}
\end{center}
We can view $P^A_{x \ot \gamma(t)}$ as a $U(n)$ valued function on $\D \times \D$ and $\omega$ as a $\C^{n \times n}$ valued function on $S\D$. Therefore, we can rewrite $I^A_{x \ot y}(\omega)$ as
\begin{equation}
	I^A_{y \ot x}(\omega) = I(x,y)\left((\pi_2^* P^A)(\pi_1^* \omega)\right).
\end{equation}
The result follows by continuity of $I$ and the pullbacks.
\end{proof}

\begin{lem}\label{lem:l2lightsink}
Let $A$ and $B$ be Hermitian connections. There is a constant $C > 0$ such that
\begin{equation}
	\norm{S^A_{z_y \ot y \ot x} - S^B_{z_y \ot y \ot x}}_{L^2(\mathcal{F}^X)} \leq C \norm{A - B}_{L^2(S\D)}.
\end{equation}
\end{lem}

\begin{proof}
By the pseudolinearisation identity and the definition of the broken attenuated X-ray, we have
\begin{equation}
	\norm{S^A_{z_y \ot y \ot x} - S^B_{z_y \ot y \ot x}}_{L^2(\mathcal{F}^X)} \leq \norm{I_{y \ot x}^{E(A,B)}(A-B)}_{L^2(\mathcal{F}^X)} + \norm{I_{z_y \ot y}^{E(A,B)}(A-B)}_{L^2(\mathcal{F}^X)}.
\end{equation}
The first term can be bounded by the same $L^2$-norm on $\D \times \D$ and is hence bounded by a multiple of $\norm{A-B}_{L^2(S\D)}$ by Lemma \ref{lem:forwardattenuated} with $k = 0$. For the second term, we have
\begin{equation}
	\norm{I^{E(A,B)}_{z_y \ot y}(A-B)}_{L^2(\mathcal{F}^X)}^2 \leq C \int_{y \in \D\setminus\mho}\int_0^T \abs{(A-B)(\dot{\gamma}_{z_y \ot y}(t))}^2 \de t \de y
\end{equation}
where $T$ is the length of the segment $\gamma_{z_y \ot y}$. We can rewrite the integral in the previous display as
\begin{equation}
	\int_{y \in \D} \int_{v \in S_y \D} \abs{(A-B)_y(v)}^2 g(y,v) \de v \de y
\end{equation}
for some bounded nonnegative function $g : S\D \to \R$. To see this, pick a point $y^* \in \D$ and a direction $v \in S_{y^*\D}$. Then, given $y \in \D \setminus \mho$, $\dot{\gamma}_{z_y \ot y}(t) = v$ for some unique $t$ precisely when $y^*$, $y$ and $z_y$ lie on the line spanned by $v$ based at $y^*$ and $y < y^* < z_y$. Since $y$ and $y^*$ must lie on a bounded line and $\D$ is bounded, it follows that $g$ is also bounded. The result readily follows by bounding $g$.
\end{proof}

\begin{lem}\label{lem:forwardtransport}
Let $A$ be a Hermitian connection. For every $k \geq 0$, there is a constant $c_k$ such that
\begin{equation}
	\norm{P^A_{y \ot x}}_{C^k(\D \times \D)} \leq c_k (1 + \norm{A}_{C^k(S\D)})^k.
\end{equation}
\end{lem}

\begin{proof}
We follow the inductive approach laid out in \cite{bohr2021stability}. Let $\Gamma$ be the diagonal of $\D \times \D$, that is, $\Gamma = \{(x,x) : x \in \D\}$. Let $\del = \del_{y \ot x}$ be the vector field on $\D \times \D \setminus \Gamma$ defined via \eqref{eq:delyx}. Then $P^A_{y \ot x}$ can be characterised as the unique smooth function $U_A$ on $\D \times \D$ such that
\begin{equation}
	(\del + A) U_A :=\del_{y \ot x} U_A(x,y) + A_y(v_{y \ot x})U_A(x,y) = 0, \quad (x,y) \in \D \times \D \setminus \Gamma,
\end{equation}
and $U_A\vert_{\Gamma} = \Id$.

Note that if $G: \D \times \D \to \C^{n \times n}$ solves
\begin{equation}
	(\del + A)G = -F, \quad (x,y) \in \D \times \D \setminus \Gamma
\end{equation}
for some $F : \D \times \D \to \C^{n \times n}$ with $G\vert_{\Gamma} = G_0$, then it follows from Lemma \ref{lem:u(T)} that
\begin{equation}
	G = P^A_{y \ot x} \left(G_0(x) - \int_0^T P^A_{x \ot \gamma(t)} F(x,\gamma(t)) \de t\right)
\end{equation}
where $\gamma : [0,T] \to \D$ is the line segment from $x$ to $y$ parametrised by arc length. Hence, it holds that
\begin{equation}\label{eq:normG}
	\norm{G}_{L^\infty(\D \times \D)} \leq C \left(\norm{G_0}_{L^\infty(\Gamma)} + \norm{F}_{L^\infty(\D \times \D \setminus \Gamma)}\right).
\end{equation}

By continuity, since $P^A_{y \ot x}$ is smooth, its $C^k(\D \times \D)$-norm agrees with its $C^k(\D \times \D \setminus \Gamma)$-norm. Let $\{\del, L_1, \dots, L_7\}$ be a global commuting frame on $\D \times \D \setminus \Gamma$ and denote $L^\alpha = L_1^{\alpha_1} \dots L_7^{\alpha_7}$ for $\alpha \in \Z^7$. Such a frame exists since we can choose global coordinates on $\D \times \D \setminus \Gamma$ by first prescribing the usual coordinates for $x$ in $\R^4$ and then choosing polar coordinates based at $x$ to describe $y$. The vector $\del$ then corresponds to the coordinate vector field related to the radial coordinate of $y$ with respect to $x$. We claim that
\begin{equation}\label{eq:Cknorms}
	\norm{P^A_{y \ot x}}_{k} := \sup_{j + \abs{\alpha} = k} \norm{\del^j L^\alpha P^A_{y \ot x}}_{L^\infty(\D \times \D\setminus\Gamma)} \lesssim_k \norm{A}_{C^k(S\D)}^k
\end{equation}
for all $k \geq 0$. For $k = 0$, this holds trivially as $P^A_{y \ot x}$ takes values in $U(n)$. Suppose now that \eqref{eq:Cknorms} holds for some $k-1 \geq 0$. Take $j$ and $\alpha$ such that $j + \abs{\alpha} = k$. Then, if we let $G := \del^j L^\alpha P^A_{y \ot x}$, we see that $G$ solves
\begin{equation}
	(\del + A)G = [A, \del^j L^\alpha]P^A_{y \ot x}, \quad (x,y) \in \D \times \D \setminus \Gamma,
\end{equation}
with $G\vert_{\Gamma} = \del^j L^\alpha P^A_{y \ot x}\vert_{y = x}$. Any differential operator $V$ on $\D \times \D$ can be decomposed as $V = V_1 + V_2$ where $V_1$ acts on the first coordinate and $V_2$ on the second with corresponding vectors $v_1$ and $v_2$ in $T\D$ based at $x$ and $y$ respectively. Lemma \ref{lem:variation} gives
\begin{equation}
	VP^A_{y \ot x} \vert_{y = x} = (P^A_{y \ot x}A_x(v_1) - A_y(v_2)P^A_{y \ot x})\vert_{y = x} = A_x(v_1 - v_2)
\end{equation}
and so we see that $\norm{G\vert_\Gamma}_{L^\infty(\Gamma)} \leq \norm{A}_{C^{k-1}(S\D)}$. Equation \eqref{eq:normG} therefore gives
\begin{equation}\label{eq:normdjLalpha}
	\norm{\del^j L^\alpha P^A_{y \ot x}}_{L^\infty(\D \times \D)} \leq C \left(\norm{A}_{C^{k-1}(S\D)} + \norm{[A, \del^j L^\alpha] P^A_{y \ot x}}_{L^\infty(\D \times \D \setminus \Gamma)}\right).
\end{equation}
The bracket $[A, \del^j L^\alpha]$ is a differential operator of order $k-1$ whose coefficients are derivatives of $A$ and can be bounded by $\norm{A}_{C^k(S\D)}$. Therefore, by the induction hypothesis, we have
\begin{equation}
	\norm{[A, \del^j L^\alpha]P^A_{y \ot x}}_{L^\infty(\D \times \D\setminus\Gamma)} \lesssim_k \norm{A}_{C^k(S\D)}\norm{P^A_{y \ot x}}_{k-1} \lesssim_k \norm{A}_{C^k(S\D)}^k.
\end{equation}
Absorbing the term for $G\vert_\Gamma$ in \eqref{eq:normdjLalpha} and taking the supremum over all $j$ and $\alpha$ such that $j + \abs{\alpha} = k$, we see that \eqref{eq:Cknorms} holds. Hence,
\begin{equation}
	\norm{P^A_{y \ot x}}_{C^k(\D \times \D)} \lesssim \sum_{j=0}^k \norm{P^A_{y \ot x}}_{j} \lesssim_k \sum_{j=0}^k \norm{A}_{C^k(S\D)}^k \lesssim_k (1 + \norm{A}_{C^k(S\D)})^k.
\end{equation}

\end{proof}

\section{Gauge invariance} \label{sec:gauge}

\subsection{Gauge invariance}

We now study the quantity $\norm{A - B - d_{E(A,B)} p}$ in \eqref{eq:pointstability} in oder to prove Theorem \ref{thm:gauge}. Recall the definition of $p^A_{\omega}$ as in \eqref{eq:p}. The $p$ used in Theorem \ref{thm:pointstability} is actually $p^{E(A,B)}_{A-B}$, which we will denote as $p_{(A,B)}$ to emphasise the dependence on $A$ and $B$ more concisely. Let us denote
\begin{equation}
	\Delta(A,B) := A - B - d_{E(A,B)} p_{(A,B)}.
\end{equation}
Recall that $\varphi \in C^\infty(\D, U(n))$ is in $\mathscr{H}$ if $\varphi\vert_{\mathcal{O}} = \Id$. The following lemma shows how $\Delta(A,B)$ changes under the action of $\mathscr{H}$.

\begin{lem} \label{lem:Theta}
Let $\Delta(A,B)$ be as above and let $\varphi, \psi \in \mathscr{H}$. Then
\begin{enumerate}
	\item[(i)] $\Delta(A,B)^* = \Delta(B,A)$,
	\item[(ii)] $\Delta(A \triangleleft \varphi, B) = \varphi^{-1} \Delta(A,B)$,
	\item[(iii)] $\Delta(A \triangleleft \varphi, B \triangleleft \psi) = \varphi^{-1} \Delta(A,B) \psi$.
\end{enumerate}
\end{lem}

\begin{proof}
We start by proving {\it (i)}. Let $u$ be the solution along the segment $\gamma : [0,T] \to \D$ from $z_y$ to $y$ of
\begin{equation}
\begin{cases}
	\dot{u} + E(A,B)(\dot{\gamma}(t))u = -(A-B)(\dot{\gamma}(t)), \\
	u(0) = 0.
\end{cases}
\end{equation}
We have
\begin{equation}
	(E(A,B)u)^* = (Au - uB)^* = u^* A^* - B^* u^* = -E(B^*, A^*) u^* = E(B,A) u^*
\end{equation}
where the last equality follows from the fact that both $A$ and $B$ are Hermitian. Hence, by taking the conjugate transpose on both sides of the equation defining $u$, we see that $u^*$ solves
\begin{equation}
\begin{cases}
	\dot{(u^*)} + E(B,A)(\dot{\gamma}(t))u^* = -(B-A)(\dot{\gamma}(t)), \\
	u^*(0) = 0.
\end{cases}
\end{equation}
It then follows from Lemma \ref{lem:u(T)} that
\begin{equation}
	p_{(A,B)}^* = -u(T)^* = p_{(B,A)}.
\end{equation}
Therefore, we can compute that
\begin{align*}
	(d_{E(A,B)} p_{(A,B)})^* &= (dp_{(A,B)} + Ap_{(A,B)} - p_{(A,B)}B)^* \\
	&= dp_{(A,B)}^* + p_{(A,B)}^* A^* - B^* p_{(A,B)}^* \\
	&= dp_{(B,A)} + B p_{(B,A)} - p_{(B,A)} A \\
	&= d_{E(B,A)} p_{(B,A)}
\end{align*}
and so
\begin{equation}
	\Delta(A,B)^* = (A - B - d_{E(A,B)} p_{(A,B)})^* = B - A - d_{E(B,A)} p_{(B,A)} = \Delta(B,A)
\end{equation}
as claimed.

Let us now prove {\it (ii)}. Let $u$ be as in the proof of {\it (i)} above and consider $\varphi^{-1} u$. We first claim that $\varphi^{-1} u$ solves
\begin{equation}
\begin{cases}
	\dot{(\varphi^{-1} u)} + E(A \triangleleft \varphi, B)(\dot{\gamma}(t))\varphi^{-1} u = -\varphi^{-1}(A-B), \\
	[\varphi^{-1}u](0) = 0.
\end{cases}
\end{equation}
Indeed,
\begin{align*}
	\dot{(\varphi^{-1}u)} + E(A \triangleleft \varphi, B)\varphi^{-1} u &= \dot{(\varphi^{-1})} u + \varphi^{-1} \dot{u} + (\varphi^{-1} \dot{\varphi})\varphi^{-1}u + (\varphi^{-1} A \varphi)\varphi^{-1} u - \varphi^{-1} u B \\
	&= \varphi^{-1} \dot{u} + \varphi^{-1} Au - \varphi^{-1}uB \\
	&= -\varphi^{-1}(A - B + E(A,B)u) + \varphi^{-1} E(A,B) u \\
	&= -\varphi^{-1}(A - B).
\end{align*}
Now notice that
\begin{equation}
	d_{E(A \triangleleft \varphi, B)} \varphi^{-1} = d(\varphi^{-1}) + \varphi^{-1} (d\varphi) \varphi^{-1} + \varphi^{-1} A \varphi \varphi^{-1} - \varphi^{-1} B = \varphi^{-1}(A - B)
\end{equation}
since $d(\varphi^{-1}) = - \varphi^{-1} (d\varphi) \varphi^{-1}$. Therefore, we can rewrite
\begin{equation}
	\varphi^{-1}(A - B) = d_{E(A \triangleleft \varphi, B)}(\varphi^{-1} - \Id + \Id) = A \triangleleft \varphi - B + d_{E(A \triangleleft \varphi, B)}(\varphi^{-1} - \Id).
\end{equation}
To make notation less cumbersome, we write $p$ for $p_{(A,B)}$ and $q$ for $p_{(A \triangleleft \varphi, B)}$. We can now apply Lemma \ref{lem:u(T)} again to get
\begin{align*}
	\varphi^{-1} p &= -\varphi^{-1} u(T) \\
	&= P^{E(A \triangleleft \varphi,B)}_{y \leftarrow z_y} I^{E(A \triangleleft \varphi,B)}_{y \leftarrow z_y}(A \triangleleft \varphi - B + d_{E(A \triangleleft \varphi, B)}(\varphi^{-1} - \Id)) \\
	&= q + P^{E(A \triangleleft \varphi,B)}_{y \leftarrow z_y} I^{E(A \triangleleft \varphi,B)}_{y \leftarrow z_y}(d_{E(A \triangleleft \varphi, B)}(\varphi^{-1} - \Id)).
\end{align*}
Notice that $\varphi^{-1} - \Id$ vanishes at $z_y$ since $\varphi\vert_{\mathcal{O}} = \Id$, and so Lemma \ref{lem:dAf} yields
\begin{equation}
	P^{E(A \triangleleft \varphi,B)}_{y \leftarrow z_y} I^{E(A \triangleleft \varphi,B)}_{y \leftarrow z_y}(d_{E(A \triangleleft \varphi, B)}(\varphi^{-1} - \Id)) = \varphi^{-1}(y) - \Id,
\end{equation}
that is, $q = \varphi^{-1} p - \varphi^{-1} + \Id$. We can now compute
\begin{align*}
	d_{E(A \triangleleft \varphi, B)} q &= dq + (A \triangleleft \varphi)q - qB \\
	&= d(\varphi^{-1}p) + (A \triangleleft \varphi) \varphi^{-1} p - \varphi^{-1}p B - d\varphi^{-1} - (A \triangleleft \varphi) \varphi^{-1} + \varphi^{-1} B \\
	& \qquad + A \triangleleft \varphi - B \\
	&= (d\varphi^{-1}) p + \varphi^{-1} dp + \varphi^{-1}(d\varphi) \varphi^{-1} p + \varphi^{-1} A p - \varphi^{-1} p B - d\varphi^{-1} \\
	& \qquad - \varphi^{-1} (d\varphi)\varphi^{-1} - \varphi^{-1} A + \varphi^{-1} B + A \triangleleft \varphi - B \\
	&= \varphi^{-1} dp + \varphi^{-1} Ap - \varphi^{-1} pB - \varphi^{-1} A + \varphi^{-1} B + A \triangleleft \varphi - B \\
	&= \varphi^{-1} d_{E(A,B)} p - \varphi^{-1}(A - B) + A \triangleleft \varphi - B
\end{align*}
and therefore
\begin{equation}
	\Delta(A \triangleleft \varphi, B) = A \triangleleft \varphi - B - d_{E(A \triangleleft \varphi, B)} q = \varphi^{-1}(A - B - d_{E(A,B)} p) = \varphi^{-1}\Delta(A,B)
\end{equation}
as claimed.

Finally, to prove {\it (iii)}, it suffices to combine {\it (i)} and {\it (ii)} to get
\begin{equation}
	\Delta(A \triangleleft \varphi, B \triangleleft \psi) = \varphi^{-1} \Delta(B \triangleleft \psi, A)^* = \varphi^{-1}(\psi^{-1} \Delta(B, A))^* = \varphi^{-1} \Delta(A,B) \psi
\end{equation}
where we used that $\psi$ takes values in $U(n)$ and therefore $\psi^{-1} = \psi^*$.
\end{proof}

\subsection{Proof of Theorem \ref{thm:gauge}}

We are now ready to prove Theorem \ref{thm:gauge}. The proof mostly amounts to using Lemma \ref{lem:Theta} with the right choice of matrix fields in $\mathscr{H}$.

\begin{proof}[Proof of Theorem \ref{thm:gauge}]
Recall that the gauge takes values in $U(n)$ and therefore if $\varphi, \psi \in \mathscr{G}$, Lemma \ref{lem:Theta} yields
\begin{equation}
	\norm{\Delta(A \triangleleft \varphi, B \triangleleft \psi)} = \norm{\varphi^{-1}\Delta(A,B)\psi} = \norm{\Delta(A,B)}.
\end{equation}
This shows that the estimate in Theorem \ref{thm:pointstability} is gauge invariant. Therefore, we can actually choose a gauge to compute $\norm{\Delta(A,B)}$. We take $\varphi = P^A_{y \ot z_y}$ and $\psi = P^B_{y \ot z_y}$.

Using that $d_{E(A,B)} \Id = A - B$, we can rewrite $p_{(A,B)}$ as
\begin{equation}\label{eq:p(A,B)}
	p_{(A,B)} = P^{E(A,B)}_{y \ot z_y} I^{E(A,B)}_{y \ot z_y}(d_{E(A,B)} \Id) = \Id - P^{E(A,B)}_{y \ot z_y}\Id = \Id - P_{y \ot z_y}^A P^B_{z_y \ot y}
\end{equation}
where we used Lemma \ref{lem:dAf} for the second equality. Hence,
\begin{align*}
\Delta(A,B) &= A - B - d_{E(A,B)}p_{(A,B)} \\
&= A - B - d_{E(A,B)}(\Id - P^{E(A,B)}_{y \ot z_y} \Id) \\
&= A - B - (A - B) + d_{E(A,B)} P^{E(A,B)}_{y \ot z_y} \Id \\
&= d_{E(A,B)} P^{E(A,B)}_{y \ot z_y} \Id.
\end{align*}
We can expand this last expression with the definition of $d_{E(A,B)}$ and Lemma \ref{lem:PE(A,B)}. This yields
\begin{equation}\label{eq:Theta}
	\Delta(A,B) = d(P^A_{y \ot z_y}P^B_{z_y \ot y}) + AP^A_{y \ot z_y}P^B_{z_y \ot y} - P^A_{y \ot z_y}P^B_{z_y \ot y}B.
\end{equation}
We have $P^{A \triangleleft \varphi}_{y \ot z_y} = \varphi^{-1}(y) P^A_{y \ot z_y} \varphi(z_y)$ from Proposition \ref{prop:gaugeparallel} and since $\varphi\vert_{\mathcal{O}} = \Id$, $\varphi(z_y) = \Id$. We chose $\varphi = P^A_{y \ot z_y}$, and so $P^{A \triangleleft \varphi}_{y \ot z_y} = \Id$. Similarly, we have $P^{B \triangleleft \psi}_{z_y \ot y} = \Id$. Therefore, plugging $A \triangleleft \varphi$ for $A$ and $B \triangleleft \psi$ for $B$ in \eqref{eq:Theta} gives
\begin{equation}
	\Delta(A \triangleleft \varphi, B \triangleleft \psi) = A \triangleleft \varphi - B \triangleleft \psi.
\end{equation}
The result readily follows.
\end{proof}

\subsection{Light-sink connections}

Recall that we called a connection light-sink if $A \triangleleft P^A_{y \ot z_y} = A$. We can characterise such connections.

\begin{prop}\label{prop:lightsink}
Let $A$ be a Hermitian connection on $\D$. Then $A \triangleleft P^A_{y \ot z_y} = A$ if and only if
\begin{equation}
	A_y\left(\del_t\right) = A_y\left(\del_r\right)
\end{equation}
for all $y \in \D$.
\end{prop}

Here, $r$ is the outward radial component in space, that is, $r^2 = x_1^2 + x_2^2 + x_3^2$. Changing to polar coordinates in the space variables, we can therefore write any light-sink connection as
\begin{equation}\label{eq:Alightsink}
	A = A_0 (dt + dr) + A_\vartheta d\vartheta + A_\phi d\phi
\end{equation}
for some matrix fields $A_0, A_\vartheta, A_\phi$ with values in $\mathfrak{u}(n)$.

\begin{proof}[Proof of Proposition \ref{prop:lightsink}]
Let $\gamma : [0,T] \to \D$ be the unit-speed lightlike geodesic from $z_y$ to $y$. Then if $A = A \triangleleft P^A_{y \ot z_y}$,
\begin{align}
	A(\dot{\gamma}(t)) &= (A \triangleleft P^A_{y \ot z_y})(\dot{\gamma}(t)) \\
	&= P^A_{z_y \ot y} [d P^A_{y \ot z_y}](\dot{\gamma}(t)) + P^A_{z_y \ot y} A(\dot{\gamma}(t)) P^A_{z_y \ot y} \\
	&= 0
\end{align}
since $[d P^A_{y \ot z_y}](\dot{\gamma}(t)) = -A(\dot{\gamma}(t)) P^A_{z_y \ot y}$ by definition of the parallel transport. In particular, by taking $t = T$, we get $A(v_{y \ot z_y}) = 0$ and so
\begin{equation}\label{eq:Adtdr}
	A_y\left(\del_t\right) = A_y\left(\del_r\right)
\end{equation}
since $v_{y \ot z_y} = \frac{1}{\sqrt{2}}\left(\del_r - \del_t\right)$. On the other hand, if \eqref{eq:Adtdr} holds for all $y \in \D$, then $P^A_{z_y \ot y} = \Id$, and so $A \triangleleft P^A_{z_y \ot y} = A$.
\end{proof}

Since $\mathscr{G}$ is a proper subgroup of $\mathscr{H}$, each $\mathscr{G}$-orbit does not necessarily contain a light-sink connection. However, Proposition \ref{prop:rho} guarantees that they are $\mathscr{H}$-equivalent to a unique light-sink connection.

\begin{proof}[Proof of Proposition \ref{prop:rho}]
Suppose that $B = A \triangleleft \varphi$ is a light-sink connection for some $\varphi \in \mathscr{H}$. Then the previous proof guarantees that
\begin{equation}
	P^{A \triangleleft \varphi}_{y \ot z_y} = \Id
\end{equation}
for all $y \in \D$ and so $\varphi(y) = P^A_{y \ot z_y}$ by Proposition \ref{prop:gaugeparallel} and the fact that $\varphi(z_y) = \Id$ because $\varphi \in \mathscr{H}$. It also follows that
\begin{equation}
	(A \triangleleft \varphi) \triangleleft P^{A \triangleleft \varphi}_{y \ot z_y} = A \triangleleft P^A_{y \ot z_y}
\end{equation}
for all $\varphi \in \mathscr{H}$ and so the map $[A] \mapsto A \triangleleft P^A_{y \ot z_y}$ is well-defined.
\end{proof}

We now show how one can retrieve a connection from its scattering data and its unique $\mathcal{H}$-equivalent light-sink connection.

\begin{proof}[Proof of Proposition \ref{prop:HtoG}]
Since $A = B \triangleleft P^B_{y \ot z_y}$, we have
\begin{equation}
	S^{A}_{z_y \ot y \ot x} = S^{B \triangleleft P^B}_{z_y \ot y \ot x} = S^B_{z_y \ot y \ot x} P^B_{x \ot z_x}
\end{equation}
for all $(x,y) \in \mathcal{F}^X(\mho)$, and so, we can determine $P^B_{y \ot z_y}$ inside $\mho^X$ from the past-determined scattering data $S^A_{z_y \ot y \ot x}$ and $S^B_{z_y \ot y \ot x}$. The same applies on $\mho^Z$ with the future-determined scattering data $S^A_{z \ot y \ot x_y}$ and $S^B_{z \ot y \ot x_y}$. Hence, we can recover $P^B_{y \ot z_y}$ on the whole of $\mho$. Let $\Phi \in C^\infty(\D, U(n))$ be any smooth extension of $P^B_{z_y \ot y}\vert_\mho$ to the whole of $\D$. Then, by \eqref{eq:HtoG},
\begin{equation}
	A \triangleleft \Phi = (B \triangleleft P^B_{y \ot z_y}) \triangleleft \Phi = B \triangleleft [P^B_{y \ot z_y}\Phi]
\end{equation}
and so $A \triangleleft \Phi$ is gauge equivalent to $B$ since $P^B_{y \ot z_y} \Phi\vert_{\mho} = \Id$.
\end{proof}

\section{Statistical application}\label{sec:stats}

We show that, when restricting ourselves to light-sink connections, one can use Bayesian inversion to consistently recover a connection from its scattering data. To do so, we follow the approach first laid out in \cite{monard2021consistent}. Specifically, we will use Theorem 5.1 in \cite{bohr2021log} as it only requires checking a nice set of conditions.

We will only consider light-sink connections as the problem then becomes injective. It would be natural to then only consider future-determined paths for the scattering data, but in doing so, the endpoints of our paths would never lie in $\mho \setminus (\mho^X \cup \mathcal{O})$. Therefore, we also need to consider past-determined paths.

To simplify the statistics slightly, we will only consider connections with values in $\mathfrak{so}(n)$ instead of $\mathfrak{u}(n)$. We do not lose any generality in doing so, but no longer have to deal with complex noise.

\subsection{Setting}

We consider the following experimental setup similar to the one we described in Section \ref{sec:statsmotivation}. Let $\lambda$ be the uniform distribution on $\mathbb{S}^+(\mho)$ induced by the Lebesgue measure on $\D$ and consider the random variables
\begin{equation}
	(X_i, Y_i, Z_i)_{i = 1}^N \sim^{i.i.d.} \lambda
\end{equation}
corresponding to random draws from $\mathbb{S}^+(\mho)$. For a light-sink Hermitian connection $A \in \Omega^1(\D, \mathfrak{so}(n))$, we denote its future-determined scattering data $S^A_{z \ot y \ot x_y}$ by $S_+^A(x,y)$ and its past-determined scattering data $S^A_{z \ot y \ot x_y}$ by $S_-^A(y,z)$. Suppose that we observe noisy versions of the scattering data corresponding to both types of paths according to our random draws, that is, we observe
\begin{equation}\label{eq:Si}
	S_i = (S_+^A(X_i, Y_i) + \mathcal{E}^+_i, S_-^A(Y_i, Z_i) + \mathcal{E}^-_i) \quad i = 1, \dots, N.
\end{equation}
The matrices $\mathcal{E}^{\pm}_i$ correspond to independent Gaussian noise in the sense that $\mathcal{E}_i^{\pm} = (\varepsilon^{\pm}_{i,j,k})_{1 \leq j,k \leq n}$ and all the $\varepsilon_{i,j,k}^{\pm}$'s are i.i.d. $N(0,1)$ that are independent from the other random variables. We denote by $P^N_A$ the joint law of the random variables $(S_i, (X_i, Y_i, Z_i))_{i=1}^N$.

In order to estimate the connection $A$ from $D_N$, we need to choose a prior $\Pi$ on the space of $\mathfrak{so}(n)$-valued light-sink connections. Any such connection can be represented by three skew-symmetric matrix fields since
\begin{equation}
	A = A_0 (dt + dr) + A_{\vartheta} d\vartheta + A_{\phi} d\phi,
\end{equation}
and therefore by $d_n := 3\dim \mathfrak{so}(n) = 3n(n-1)/2$ continuous functions on $\D$. Following \cite{bohr2021log}, we choose the prior $\Pi$ by prescribing an orthonormal basis on $L^2(\D, \R)$ as well as a sequence of positive scalars. For conciseness, we choose as basis the normalised eigenfunctions $(e_j)_{j \in \N}$ of the Laplacian with Neumann boundary conditions and choose their eigenvalues $(\lambda_j)_{j \in \N}$ as scalars. It follows from classical $L^\infty$ estimates for eigenfunctions from \cite{hormander1968spectral} and Weyl's law \cite{hormanderiv} that we can choose $\tau = 3/4$ and $d = 4$ in Condition 3.1 of \cite{bohr2021log}. This choice gives rise to Sobolev-type spaces
\begin{equation}
	H^s(\D, \R) = \left\{f \in L^2(\D, \R) : \sum_{j \in \N} \lambda_j^s \dual{f}{e_j}_{L^2(\D, \R)}^2 < \infty \right\}
\end{equation}
which, in this specific case, agree with the usual Sobolev spaces. These Sobolev spaces naturally induce Sobolev spaces on the space of light-sink connections 
\begin{equation}
	H^s(\D, \mathfrak{so}(n)^3) \simeq H^s(\D, \R^{d_n}) = \bigtimes_{i=1}^{d_n} H^s(\D, \R).
\end{equation}
which plays the role of our parameter space $\Theta$. The eigenfunctions $(e_j)$ naturally induce the basis $\{e_{j,i} : 1 \leq i \leq d_n, j \in \N \}$ on $H^s(\D, \mathfrak{so}(n)^3)$ where
\begin{equation}
	e_{j,i} = (\delta_{i,1} e_j, \dots, \delta_{i,d_n} e_j), \quad \delta_{i,j} =
	\begin{cases}
		1 & i = j, \\
		0 & i \neq j.
	\end{cases}
\end{equation}
For $D$, an integer multiple of $d_n$, let $E_D$ be the span of the first $D$ vectors of the basis, that is,
\begin{equation}
	E_D := \{e_{j,i} : 1 \leq i \leq d_n, 1 \leq j \leq D/d_n \}.
\end{equation}
For $\alpha > 0$, we take as prior on $E_D$
\begin{equation}\label{eq:prior}
A = N^{-1/(\alpha + 2)} \sum_{i \leq d_n} \sum_{j \leq D/d_n} \lambda_n^{-\alpha/2} g_{j,i} e_{j,i}, \quad g_{j,i} \sim^{i.i.d.} \mathcal{N}(0,1).
\end{equation}
Theorem \ref{thm:stats} can also be proved for $D \to \infty$, giving rise to a commonly used Matérn prior of order $\alpha$ for the Laplacian, see e.g. \cite[Chapter 11]{ghosal} . For simplicity, we take a truncated prior as it reflects what happens in practice. We denote the law of $A$ by $\Pi$ and its density by $\pi$. Through Bayes' rule, the choice of prior gives rise to the posterior distribution 
\begin{equation}
	\Pi(A \in O \vert (S_i, (X_i, Y_i, Z_i))_{i=1}^N) = \frac{\int_{O} e^{\ell_N(A)} d\Pi(A)}{\int_{E_D} e^{\ell_N(A)} d\Pi(A)}, \quad O \subseteq E_D \text{ Borel}
\end{equation}
with log-likelihood given by
\begin{equation}
\ell_N(A) = -\frac{1}{2} \sum_{i=1}^N \abs{S_i - S(A)(X_i, Y_i, Z_i)}_{(\R^{n\times n})^2} - \frac{1}{2} N\delta_N^2\norm{A}_{H^\alpha}^2
\end{equation}
up to some additive constant, where $\delta_N = N^{-\alpha/(2\alpha + 4)}$. See \cite{bohr2021log} for more details.

\subsection{Statistical guarantees for light-sink connections}

To apply \cite[Theorem 5.1]{bohr2021log}, it remains to show that the map $S : A \mapsto (S^A_{+}, S^A_{-})$ satisfies their Condition 3.2 that contains three parts. The first part, uniform boundedness, is immediately satisfied since $S^A$ takes values in $U(n)$. The second part consists of global Lipschitz estimates for the forward map and follows from Lemma \ref{lem:forward} in the $L^\infty$ case and from Lemma \ref{lem:l2lightsink} in the $L^2$ case. Hence, it only remains to show that the last part of their Condition holds, which they have called \textit{inverse continuity modulus}. That is the content of the next lemma.

\begin{lem}\label{lem:condition32}
For every $M$ there exists a constant $L'$ and $0 < \gamma \leq 1$ such that for all $\delta > 0$ small enough and the given $\alpha > 0$,
\begin{equation}
	\sup\left\{\norm{A - B}_{L^2(\D)} : \norm{A}_{H^\alpha(\D)} + \norm{B}_{H^\alpha(\D)} \leq M, \norm{S(A) - S(B)}_{L^2(\mathbb{S}^+(\mho))} \leq \delta \right\} \leq L' \delta^\gamma.
\end{equation}
\end{lem}

\begin{proof}
First, note that we can find $C_{\varepsilon} > 0$ that depends on $\varepsilon$ such that
\begin{equation}
	\norm{S(A) - S(B)}_{L^2(\mathbb{S}^+(\mho))}^2 \geq C_\varepsilon \left( \norm{S^A_+ - S^B_+}_{L^2(\mathcal{F}^X)}^2 + \norm{S^A_- - S^B_-}_{L^2(\mathcal{F}^Z)}^2\right).
\end{equation}
By the interpolation inequality for Sobolev spaces (see \cite[Lemma 7.4]{bohr2021stability}), for every $s > 0$,
\begin{equation}
	\norm{\,\cdot\,}_{H^1(\mathcal{F}^X)} \lesssim_s \norm{\,\cdot\,}_{L^2(\mathcal{F}^X)}^{1 - 1/s} \norm{\,\cdot\,}_{H^s(\mathcal{F}^X)}^{1/s}.
\end{equation}
By applying this inequality to $[S^A_+]^{-1}S^B_+ - \Id$ and using the pseudolinearisation identity, we get
\begin{equation}
	\norm{[S^A_+]^{-1}S^B_+ - \Id}_{H^1(\mathcal{F}^X)} \lesssim_s \norm{S^A_+ - S^B_+}_{L^2(\mathcal{F}^X)}^{1 - 1/s} \norm{I^{E(A,B)}_{z_y \ot y \ot y}(A-B)}_{H^s(\mathcal{F}^X)}^{1/s}.
\end{equation}
Since $A$ and $B$ are light-sink connections, the broken attenuated X-ray is equal to the simple attenuated X-ray $I^{E(A,B)}_{y \ot x}(A-B)$. For $s = k$ an integer, Lemmas \ref{lem:forwardattenuated} and \ref{lem:forwardtransport} yield
\begin{equation}
	\norm{I^{E(A,B)}_{y \ot x}(A-B)}_{H^k(\mathcal{F}^X)} \lesssim_k \norm{A-B}_{H^k(\mathcal{F}^X)}(1 + \norm{E(A,B)}_{C^k(S\D)})^k.
\end{equation}
Taking $\alpha > k$ sufficiently large such that $\norm{A}_{H^\alpha(S\D)} + \norm{B}_{H^\alpha(S\D)} \leq M$, we can bound the $C^k$ norms of $A$ and $B$ via Sobolev embedding inequalities. Note that $\alpha = k + 3$ suffices. This in turn allows us to bound the $C^k$-norm of $E(A,B)$, and so
\begin{equation}
	\norm{[S^A_+]^{-1}S^B_+ - \Id}_{H^1(\mathcal{F}^X)} \lesssim_{k,M} \norm{S^A_+ - S^B_+}_{L^2(\mathcal{F}^X)}^{1 - 1/k}.
\end{equation}
It follows from Theorem \ref{thm:estimateout} for light-sink connections ($p = 0$ by Theorem \ref{thm:gauge}) that
\begin{equation}
	\norm{A - B}_{L^2(\D\setminus\mho)} \lesssim_{\varepsilon, k, M} \norm{S^A_+ - S^B_+}^{1 - 1/k}_{L^2(\mathcal{F}^X)}.
\end{equation}
Similar inequalities also hold on $\mho^X$ and $\mho^Z$ by Theorem \ref{thm:estimatein}. Hence, we can choose $\gamma = \frac{\alpha - 4}{\alpha - 3}$ for $\alpha \geq 5$ by taking $\alpha = k+3$.
\end{proof}

We can finally apply Theorem 5.1 in \cite{bohr2021log} to get the following estimate regarding the concentration of the posterior distribution around the real parameter $A_\star$ obtained through noisy samples of $S^{A_\star}$ as the number of samples goes to infinity.
\begin{thm}\label{thm:stats}
Let the posterior distribution $\Pi(\cdot \vert (S_i, (X_i, Y_i, Z_i))_{i=1}^N)$ arise from the prior \eqref{eq:prior} with $\alpha \geq 5$ and data $(S_i, (X_i, Y_i, Z_i)) \sim P_A^N$ as in \eqref{eq:Si}. Suppose that $A_\star \in H^\alpha$ and $D \simeq N^{2/(\alpha + 2)}$. Let $\gamma = \frac{\alpha - 4}{\alpha - 3}$. Then, there is $M > 0$ such that
\begin{equation}
	\Pi\left(\norm{A - A_\star}_{L^2(\D)} > M\delta_N^\gamma \vert (S_i, (X_i, Y_i, Z_i))_{i=1}^N\right) = o_{P^N_{A_\star}}(1)
\end{equation}
as $N \to \infty$, where $\delta_N = N^{-\alpha/(2\alpha+4)}$.
\end{thm}

In short, the posterior distribution converges to a delta distribution about $A_\star$ in $P^N_{A_\star}$-probability at a rate that depends on the smoothness of the prior and of $A_\star$. The smoother $A_\star$ is, the smoother we can choose the prior, and the faster the posterior distribution concentrates. Moreover, by the same arguments used at the end of \cite{monard2021consistent} to complete their proof of their Theorem 3.2, one can expect the rate of Theorem \ref{thm:stats} to carry over to the posterior mean, that is,
\begin{equation}
	\norm{E^\Pi[A\vert (S_i, (X_i, Y_i, Z_i))_{i=1}^N)] - A_\star }_{L^2(\D)} = O_{P^N_{A_\star}}(\delta_N^\gamma)
\end{equation}
as $N \to \infty$.

Note that we have convergence to $A_\star$ in Theorem \ref{thm:stats} and not only its projection $A_{\star, D}$ on $E_D$ as in the statement of Theorem 5.1 in \cite{bohr2021log}. This is due to the fact that the estimate in Lemma \ref{lem:condition32} holds for all $A$, $B$ in the whole parameter space $H^\alpha(\D, \mathfrak{so}(n)^3)$, and not just $E_D$. Indeed, Remark 5.2 in \cite{bohr2021log} guarantees that we can then replace $A_{\star, D}$ by $A_\star$.

\bibliographystyle{alpha}

\bibliography{reference}

\end{document}